\documentclass[a4paper,reqno,11pt]{amsart}
\usepackage{amsmath}
\usepackage{amssymb}
\usepackage{bbm}
\usepackage{hyperref}
\usepackage[numbers,sort&compress]{natbib}
\usepackage{amscd}

\usepackage[left=3.2cm,right=3.2cm,bottom=3cm,top=3cm]{geometry}

\newtheorem{thm}{Theorem}[section]
\newtheorem{lemma}[thm]{Lemma}
\newtheorem{prop}[thm]{Proposition}
\newtheorem{cor}[thm]{Corollary}
\newtheorem{question}[thm]{Question}
\newtheorem{fact}[thm]{Fact}

\theoremstyle{definition}

\newtheorem{nrmk}[thm]{Remark}

\newtheorem*{rmk}{Remark}

\theoremstyle{remark}

\makeatletter

\def\dotminussym#1#2{%
  \setbox0=\hbox{$\m@th#1-$}%
  \kern.5\wd0%
  \hbox to 0pt{\hss\hbox{$\m@th#1-$}\hss}%
  \raise.6\ht0\hbox to 0pt{\hss$\m@th#1.$\hss}%
  \kern.5\wd0}
\newcommand{\dotminus}{\mathbin{\mathpalette\dotminussym{}}}

\newcommand{\Z}{\mathbb{Z}}

\newcommand{\curly}[1]{\mathcal{#1}}

\newcommand{\B}{\curly{B}}

\newcommand{\n}{\mathbb{N}}

\newcommand{\la}{\curly{L}}

\renewcommand{\to}{\rightarrow}

\def \<{\langle}
\def \>{\rangle}

\def \*Z {{{^*}\Z}}
\def \((  {(\!(}
\def \)) {)\!)}

\def \O{\operatorname{O}}

\def \tp{\operatorname{tp}}

\numberwithin{equation}{section}

\def \Th{\operatorname{Th}}

\def \O{\mathcal O}

\def \ng{\mathbb{NG}}
\def \gs{\mathbb{GS}}

\allowdisplaybreaks[2]

\title[Model-theoretic aspects of the Gurarij operator system]{Model-theoretic aspects\\ of the Gurarij operator system}
\author{Isaac Goldbring and Martino Lupini}
\thanks{Goldbring's work was partially supported by NSF CAREER grant DMS-1349399.  Lupini's work was supported by the York University Susan Mann Dissertation Scholarship and by the ERC Starting
grant no.\ 259527 of Goulnara Arzhantseva. This
work was initiated during a visit of the second author to the University of Illinois at Chicago. The hospitality of the UIC Mathematics Department is gratefully
acknowledged.}

\address {Isaac Goldbring, Department of Mathematics, Statistics, and Computer Science, University of Illinois at Chicago, Science and Engineering Offices M/C 249, 851 S. Morgan St., Chicago, IL, 60607-7045}
\email{isaac@math.uic.edu}
\urladdr{http://www.math.uic.edu/~isaac}

\address{Fakult\"{a}t f\"{u}r Mathematik,
Universit\"{a}t Wien,
Oskar-Morgenstern-Platz 1,
Room 02.126,
1090 Wien, Austria.}
\email{martino.lupini@univie.ac.at}
\urladdr{http://www.lupini.org/}

\begin{document}

\begin{abstract}
We establish some of the basic model theoretic facts about the Gurarij operator system $\mathbb{GS}$ recently constructed by the second-named author.  In particular, we show:  (1) $\mathbb{GS}$ is the unique separable 1-exact existentially closed operator system; (2) $\mathbb{GS}$ is the unique separable nuclear model of its theory; (3) every embedding of $\mathbb{GS}$ into its ultrapower is elementary; (4) $\mathbb{GS}$ is the prime model of its theory; and (5) $\mathbb{GS}$ does not have quantifier-elimination, whence the theory of operator systems does not have a model companion.  We also show that, for any $q\in \n$, the theories of $M_q$-spaces and $M_q$-systems do have a model companion, namely the Fra\"{i}ss\'{e} limit of the class of finite-dimensional $M_q$-spaces and $M_q$-systems respectively; moreover we show that the model companion is separably categorical.  We conclude the paper by showing that no C$^*$ algebra can be existentially closed as an operator system.
\end{abstract}

\maketitle

\section{Introduction}

The Gurarij Banach space $\mathbb{G}$ is a Banach space first constructed by Gurarij in \cite{gurarij_spaces_1966}.  It has the following universal property:  whenever $X\subseteq Y$ are finite-dimensional Banach spaces, $\phi:X\to \mathbb{G}$ is a linear isometry, and $\epsilon>0$, there is an injective linear map $\psi:Y\to \mathbb{G}$ extending $\phi$ such that $\Vert\psi \Vert \Vert\psi^{-1}\Vert<1+\epsilon$.  The uniqueness of such a space was first proved by Lusky in \cite{lusky_gurarij_1976} and later a short proof was given by Kubis and Solecki in \cite{kubis_proof_2013}.

Model-theoretically, $\mathbb{G}$ is a relatively nice object.  Indeed, Ben Yaacov \cite{ben_yaacov_fraisse_2012} showed that $\mathbb{G}$ is the Fra\"iss\'{e} limit of the class of finite-dimensional Banach spaces (yielding yet another proof of the uniqueness of $\mathbb{G}$).  Moreover, Ben Yaacov and Henson \cite{ben_yaacov_generic_2012} showed that the theory of $\mathbb{G}$ is separably categorical and admits quantifier-elimination; since every separable Banach space embeds in $\mathbb{G}$, it follows that the theory of $\mathbb{G}$ is the model-completion of the theory of Banach spaces.  (On the other hand, it is folklore that the theory of $\mathbb{G}$ is unstable, so perhaps the nice model-theoretic properties of $\mathbb{G}$ end here.)

In \cite{oikhberg_non-commutative_2006}, Oikhberg introduced a noncommutative analog of $\mathbb{G}$ which he referred to as (no surprise) a \emph{noncommutative Gurarij operator space}.  Here, ``noncommutative'' refers to the fact that we are considering \emph{operator spaces}, the noncommutative analog of Banach spaces.  (In Section 2, a primer on operator spaces---amongst other things---will be given.)  A Gurarij operator space satisfies the noncommutative analog of the defining property of $\mathbb{G}$ mentioned above, where the completely bounded norm replaces the usual norm of linear maps and the finite-dimensional spaces are further assumed to be $1$-exact.  Approximate uniqueness of a Gurarij operator space was proven by Oikhberg in \cite{oikhberg_non-commutative_2006}. 

Precise uniqueness of the Gurarij operator space was later proven in \cite{lupini_uniqueness_2014} by realizing \emph{the} Gurarij operator space (henceforth referred to as $\ng$) as the Fra\"iss\'{e} limit of the class of finite-dimensional 1-exact operator spaces.  In \cite{lupini_universal_2014}, the second author established the existence and uniqueness of the \emph{Guarij operator system} $\mathbb{GS}$, which is the Fra\"{i}s\'{e} limit of the class of finite-dimensional 1-exact operator systems.  In this paper, we establish some of the main model-theoretic facts about $\mathbb{GS}$.  We summarize these as follows:

\begin{thm}

\

\begin{enumerate}
\item $\mathbb{GS}$ is the unique separable 1-exact existentially closed operator system; 
\item $\mathbb{GS}$ is the unique separable nuclear model of its theory; 
\item every embedding of $\mathbb{GS}$ into its ultrapower is elementary; 
\item $\mathbb{GS}$ is the prime model of its theory; 
\item $\mathbb{GS}$ does not have quantifier-elimination, whence the theory of operator systems does not have a model companion.
\end{enumerate}
\end{thm}

We do not know whether the analogous statements hold for $\mathbb{NG}$. The proofs in the operator system case rely strongly on Choi's characterization of completely positive linear maps with domain $\mathbb{M}_n$ \cite{choi_completely_1975}; see also \cite[Theorem 3.14]{paulsen_completely_2002}. Such a characterization has no counterpart for (not necessarily unital) completely bounded maps, as pointed out in \cite[page 114]{paulsen_completely_2002}.

The notion of $M_q$-space is an intermediate generalization of Banach spaces where only matrix norms up to order $q$ are considered; likewise, there is the notion of $M_q$-system (also called $q$-minimal operator system).  The classes of finite-dimensional $M_q$-spaces (resp.\ $M_q$-systems) have Fr\"{i}ss\'{e} limits, called $\mathbb{G}_q$ (resp.\ $\mathbb{G}_q^u$).  We show that the theory of $\mathbb{G}_q$ (resp.\ the theory of $\mathbb{G}_q^u$) is the model-completion of the theory of $M_q$-spaces (resp.\ $M_q$-systems) and is separably categorical.

We conclude the paper by proving that no existentially closed operator system can be completely order isomorphic to a C*-algebra.  While this fact was proven for $\mathbb{GS}$ itself in \cite{lupini_universal_2014}, our proof here is somewhat more elementary and covers all existentially closed operator systems.

We assume that the reader is familiar with continuous logic as it pertains to operator algebras (see \cite{farah_model_2015} for a good primer). In the following all the quantifiers are supposed to range in the unit ball. If $X$ is an operator space or an operator system, its closed unit ball is denoted by $\mathrm{Ball}(X)$. In Section 2, we describe all of the necessary background on operator spaces and operator systems.

\subsection*{Acknowledgments} We would like to thank Timur Oikhberg for pointing out a mistake in the original version of this paper.

\section{Preliminaries}

\subsection{Operator spaces and \texorpdfstring{$M_q$}{Mq}-spaces}

If $H$ is a Hilbert space, let $B\left( H\right) $ denote the space of bounded
linear operators on $H$ endowed with the pointwise linear operations and the
operator norm. One can identify $M_{n}\left( B\left( H\right) \right) $ with
the space $B\left( H^{\oplus n}\right) $, where $H^{\oplus n}$ is the $n$%
-fold Hilbertian sum of $H$ with itself. A \emph{(concrete) operator space} is a
closed subspace of $B\left( H\right) $. If $X$ is an operator space, then
the inclusion $M_{n}\left( X\right) \subset M_{n}\left( B\left( H\right)
\right) $ induces a norm on $M_{n}\left( X\right) $ for every $n\in \mathbb{N%
}$. If $X,Y$ are operator spaces, $\phi :X\rightarrow Y$ is a linear map,
and $n\in \mathbb{N}$, then the $n$-th amplification $\phi ^{\left( n\right)
}:M_{n}\left( X\right) \rightarrow M_{n}\left( Y\right) $ is defined by%
\begin{equation*}
\left[ x_{ij}\right] \mapsto \left[ \phi \left( x_{ij}\right) \right] \text{.%
}
\end{equation*}%
A linear map $\phi $ is \emph{completely bounded} if $\sup_{n}\left\Vert
\phi ^{\left( n\right) }\right\Vert <+\infty $, in which case one defines
the completely bounded norm $\left\Vert \phi \right\Vert
_{cb}:=\sup_{n}\left\Vert \phi ^{\left( n\right) }\right\Vert $. We say that $\phi $ is \emph{completely contractive} if $\phi ^{\left( n\right) }$ is contractive for every $n\in \mathbb{N}$ and \emph{completely isometric} if $\phi ^{\left( n\right) }$ is isometric for every $n\in \mathbb{N}$.

If $q\in \mathbb{N}$, $\alpha ,\beta \in M_{q}$, and $x\in
M_{q}\left( X\right) $ we denote by $\alpha .x.\beta $ the element of $%
M_{q}\left( X\right) $ obtained by taking the usual matrix product.  The matrix norms on an operator space satisfy the following relations, known
as Ruan's axioms:  for every $q,k\in \mathbb{N}$ and $x\in M_{q}\left( X\right) $ we have %
\begin{equation*}
\left\Vert 
\begin{bmatrix}
x & 0 \\ 
0 & 0%
\end{bmatrix}%
\right\Vert _{M_{q+k}\left( X\right) }=\left\Vert x\right\Vert _{M_{q}\left(
X\right) }\text{,}
\end{equation*}%
and for every $q,n\in \mathbb{N}$, $\alpha _{i},\beta _{i}\in M_{q}$ and $%
x_{i}\in M_{q}\left( X\right) $ for $i=1,2,\ldots ,n$, we have
\begin{equation}
\left\Vert \sum_{i=1}^{n}\alpha _{i}.x_{i}.\beta _{i}\right\Vert \leq
\left\Vert \sum_{i=1}^{n}\alpha _{i}\alpha _{i}^{\ast }\right\Vert
\max_{1\leq i\leq n}\left\Vert x_{i}\right\Vert \left\Vert
\sum_{i=1}^{n}\beta _{i}^{\ast }\beta _{i}\right\Vert \text{\label%
{Equation:(R)},}
\end{equation}%
where $\alpha _{i}.x_{i}.\beta _{i}$ denotes the usual matrix multiplication. Ruan's theorem \cite{ruan_subspaces_1988} asserts that, conversely, any
matricially normed vector space $X$ with matrix norms satisfying Ruan's
axioms is linearly completely isometric to a subspace of $B\left(
H\right) $; see also \cite[\S 2.2]{pisier_introduction_2003}. We will
regard operator spaces as structures in the language for operator spaces $\la_{ops}$ introduced in \cite[Appendix B]{goldbring_kirchbergs_2014}. It is clear that the class of operator spaces, viewed as $\la_{ops}$-structures, forms an axiomatizable class by semantic considerations \cite[\S 2.3.2]{farah_model_2015}.  Using Ruan's theorem, concrete axioms for the class of operator spaces are given in \cite[Theorem B.3]{goldbring_kirchbergs_2014}.

A finite-dimensional operator space $X$ is said to be \emph{1-exact} if there are natural numbers $k_n$ and injective linear maps $\phi_n:X\to M_{k_n}$ such that $\|\phi_n\|_{cb}\|\phi_n^{-1}\|_{cb}\to 1$ as $n\to \infty$. An arbitrary operator space is $1$-exact if all its finite-dimensional subspaces are $1$-exact. It is well known that a C*-algebra is exact if and only if it is 1-exact when viewed as an operator space.  We mention in passing that the 1-exact operator spaces do not form an axiomatizable class, even amongst the separable ones.  

For $q\in \n$, an \emph{$M_{q}$-space} is a vector space $X$ such that $M_{q}\left( X\right) $ is
endowed with a norm satisfying Equation (\ref{Equation:(R)}) for every $n\in 
\mathbb{N}$, $\alpha _{i},\beta _{i}\in M_{q}$, and $x_{i}\in M_{q}\left(
X\right) $ for $i=1,2,\ldots ,n$. Clearly an $M_{q}$-space is canonically an 
$M_{n}$-space for $n\leq q$ via the upper-left corner embedding of $%
M_{n}\left( X\right) $ into $M_{q}\left( X\right) $. Let $T_{M_{q}}$ be the
reduct of the language of operator spaces where only the sorts for $M_{n,m}$
for $n,m\leq 2q$ and $M_{n,m}\left( X\right) $ for $n,m\leq q$ are retained.
Once again, by syntactic considerations, it is straightforward to verify that $M_{q}$-spaces form an axiomatizable class in the
language $T_{M_{q}}$. One can write down explicit axioms using Equation (\ref{Equation:(R)}).

If $\phi :X\rightarrow Y$ is a linear map between $M_{q}$-spaces, then $\phi 
$ is said to be \emph{$q$-bounded} if $\phi ^{\left( q\right) }:M_{q}\left( X\right)
\rightarrow M_{q}\left( X\right) $ is bounded. In such a case one sets $%
\left\Vert \phi \right\Vert _{q}=\left\Vert \phi ^{\left( q\right)
}\right\Vert $. A linear map $\phi $ is then said to be \emph{$q$-contractive} if $\phi ^{\left(q\right)}$ is contractive and \emph{$q$-isometric} if $\phi ^{\left(q\right)}$ is isometric.

It is shown in \cite[Th\'{e}or\`{e}me I.1.9]{lehner_mn-espaces_1997} that
any $M_{q}$-space can be concretely represented as a subspace of $C\left(
K,M_{q}\right) $ for some compact Hausdorff space $K$. Here $C\left(
K,M_{q}\right) $ is the space of continuous functions from $K$ to $M_{q}$
endowed with the $M_{q}$-space structure obtained by canonically identifying 
$M_{q}\left( C\left( K,M_{q}\right) \right) $ with $C\left( K,M_{q}\otimes
M_{q}\right) $, where the latter is endowed with the uniform norm.

An $M_{q}$-space $X$ admits a canonical operator space structure denoted by $%
\mathrm{MIN}_{q}\left( X\right) $ \cite[I.3]{lehner_mn-espaces_1997}. The
corresponding operator norms are defined by%
\begin{equation*}
\left\Vert x\right\Vert =\sup_{\phi }\left\Vert \phi ^{\left( q\right)
}\left( x\right) \right\Vert
\end{equation*}%
for $n\in \mathbb{N}$ and $x\in M_{n}\left( X\right) $, where $\phi $ ranges
over all $q$-contractive linear maps $\phi :X\rightarrow M_{q}$.  The $\mathrm{MIN}_q$ operator space structure on $X$ is characterized by the following property:  the identity map $%
X\rightarrow \mathrm{MIN}_{q}\left( X\right) $ is a $q$-isometry, and for
any operator space $Y$ and linear map $\phi :Y\rightarrow X$, the map $\phi $
is $q$-bounded if and only if $\phi :Y\rightarrow \mathrm{MIN}_{q}\left(
X\right) $ is completely bounded, in which case $\left\Vert \phi
:Y\rightarrow X\right\Vert _{q}=\left\Vert \phi :Y\rightarrow \mathrm{%
\mathrm{MI}N}_{q}\left( X\right) \right\Vert $.

We will call an operator space of the form $\mathrm{MIN}_{q}\left( X\right) $
a $\mathrm{MIN}_{q}$-\emph{space}. It is clear that semantically there is
really no difference between $M_{q}$-spaces and $\mathrm{MIN}_{q}$-spaces.
However there is a syntactical difference between these two notions as
they correspond to regarding these spaces as structures in two different
languages. We will therefore retain the two distinct names to avoid
confusion.

It follows from the characterizing property of the functor $\mathrm{MIN}_{q}$ that $\mathrm{%
\mathrm{MI}N}_{q}$-spaces are closed under subspaces, isomorphism, and
ultraproducts. (For the latter, one needs to observe that the ultraproduct
of a family of $q$-bounded maps from $X$ to $M_{q}$ is again a $q$-bounded
map from $X$ to $M_{q}$.) Therefore, $\mathrm{MIN}%
_{q}$-spaces form an axiomatizable class in the language of operator spaces.  Furthermore the functor $\mathrm{MIN}_{q}$ is an
equivalence of categories from $M_{q}$-spaces to $\mathrm{MIN}_{q}$-spaces.  It follows from Beth's definability theorem \cite[\S 3.4]{farah_model_2015} the that the matrix norms on $%
M_{n}\left( X\right) $ for $n>q$ are definable in the language of $M_{q}$%
-spaces.

\subsection{Operator systems and \texorpdfstring{$M_q$}{Mq}-systems}

Suppose that $X$ is an operator space. An element $u\in X$ is a \emph{%
unitary }if there is a linear complete isometry $\phi :X\rightarrow B\left(
H\right) $ such that $\phi \left( u\right) $ is the identity operator on $H$%
. It is shown in \cite{blecher_metric_2011} that if $X$ is a C*-algebra, then this corresponds with the usual notion of unitary.  Theorem 2.4 of \cite{blecher_metric_2011} provides the following abstract
characterization of unitaries:   $u$ is a unitary of $X$ if an only if, for
every $n\in \mathbb{N}$ and $x\in M_{n}\left( X\right) $, one has that%
\begin{equation*}
\left\Vert 
\begin{bmatrix}
u_n & x%
\end{bmatrix}%
\right\Vert ^{2}=\left\Vert 
\begin{bmatrix}
u_{n} \\ 
x%
\end{bmatrix}%
\right\Vert ^{2}=1+\left\Vert x\right\Vert ^{2}\text{,}
\end{equation*}%
where $u_{n}$ denotes the diagonal matrix in $M_{n}\left( X\right) $ with $u$ in the diagonal entries.
A \emph{unital operator space} is an operator space with a distinguished unitary.
The abstract characterization of unitaries shows that unital operator spaces
form an axiomatizable class in the language $T_{uosp}$ obtained by adding to
the language of operator spaces a constant symbol for the unit.

If $X$ is an $M_{q}$-space, then we say that an element $u$ of $M_{q}$ is a unitary if
there is a linear $q$-isometry $\phi :X\rightarrow C\left(
K,M_{q}\right) $ mapping the distinguished unitary to the function constantly equal to the identity of $M_q$. Observe that $u$ is a unitary of $X$ if and only if it is
a unitary of $\mathrm{MIN}_{q}\left( X\right) $. In fact, if $u$ is a unitary
of $X$ and $\phi :X\rightarrow C\left( K,M_{q}\right) $ is a unital linear $q$%
-isometry, then $\phi :\mathrm{MIN}_{q}\left( X\right) \rightarrow C\left(
K,M_{q}\right) $ is a unital complete isometry. If $\psi :C\left(
K,M_{q}\right) \rightarrow B\left( H\right) $ is a unital complete isometry,
then $\phi \circ \psi $ witnesses the fact that $u$ is a unitary of $\mathrm{%
MIN}_{q}\left( X\right) $. Conversely suppose that $u$ is a unitary of $%
\mathrm{MIN}_{q}\left( X\right) $. It follows from the universal property
that characterizes the injective envelope of an operator space \cite[\S
4.3]{blecher_operator_2004} that the injective envelope $I\left( \mathrm{MIN}%
_{q}\left( X\right) \right) $ is a $\mathrm{MIN}_{q}$-space. Since the
C*-envelope $C_{e}^{\ast }\left( \mathrm{MIN}_{q}\left( X\right) ,u\right) $
of the unital operator space $\mathrm{MIN}_{q}\left( X\right) $ with unit $u$
can be realized as a subspace of $I\left( \mathrm{MIN}_{q}\left( X\right)
\right) $ by \cite[\S 4.3]{blecher_operator_2004}, it follows that $%
C_{e}^{\ast }\left( \mathrm{MIN}_{q}\left( X\right) ,u\right) $ is an $M_{q}$%
-space. Equivalently $C_{e}^{\ast }\left( \mathrm{MIN}_{q}\left( X\right)
,u\right) $ is a $q$-subhomogeneous C*-algebra \cite[IV.1.4.1]%
{blackadar_operator_2006}. Therefore there is an injective unital
*-homomorphism $\psi :C_{e}^{\ast }\left( \mathrm{MIN}_{q}\left( X\right)
,u\right) \rightarrow \bigoplus_{k\leq q}C\left( K,M_{k}\right) $
for some compact Hausdorff space $K$; see \cite[IV.1.4.3]%
{blackadar_operator_2006}.

Moreover the proof of \cite[Theorem 2.4]{blecher_metric_2011} shows that an
element $u$ of an $M_{q}$-space $X$ is a unitary if and only if%
\begin{equation*}
\left\Vert 
\begin{bmatrix}
u_{q} & x%
\end{bmatrix}%
\right\Vert _{M_{2q}\left( \mathrm{\mathrm{MIN}}_{q}\left( X\right) \right)
}^{2}=\left\Vert 
\begin{bmatrix}
u_{q} \\ 
x%
\end{bmatrix}%
\right\Vert _{M_{2q}\left( \mathrm{MIN}_{q}\left( X\right) \right)
}^{2}=1+\left\Vert x\right\Vert ^{2}\text{.}
\end{equation*}%
A \emph{unital $M_{q}$-space} is an $M_{q}$-space with a distinguished unitary. Let 
$T_{uM_{q}}$ the language of $M_{q}$-spaces with an additional constant
symbol for the distinguished unitary. Then the abstract characterization of
unitaries in $M_{q}$-spaces provided above together with the fact that the
matrix norms on $\mathrm{MIN}_{q}\left( X\right) $ are definable show that
unital $M_{q}$-spaces form an axiomatizable class in the language of unital $%
M_{q}$-spaces.

An \emph{operator system }is a unital operator space $\left( X,1\right) $
such that there exists a linear complete isometry $\phi :X\rightarrow
B\left( H\right) $ with $\phi \left( 1\right) =1$ and $\phi \left[ X\right] $
a self-adjoint subspace of $B\left( H\right) $. By \cite[Theorem 3.4]%
{blecher_metric_2011}, a unital operator space is an operator system if and
only if for every $n\in \mathbb{N}$ and for every $x\in X$ there is $y\in Y$
such that $\left\Vert y\right\Vert \leq \left\Vert x\right\Vert $ and%
\begin{equation}
\left\Vert 
\begin{bmatrix}
n1 & x \\ 
y & n1%
\end{bmatrix}%
\right\Vert ^{2}\leq 1+n^{2}\text{.\label{Equation:adjoint}}
\end{equation}%
This shows that operator systems form an axiomatizable class in the language
of unital operator spaces.

The representation of an operator system $X$ as a unital self-adjoint subspace of $X$ induces
on $X$ an involution $x\mapsto x^{\ast }$ and positive cones on $M_{n}\left(
X\right) $ for every $n\in \mathbb{N}$. A linear map between operator systems is \emph{positive} if it maps positive elements to positive elements, and \emph{completely positive} if all its amplifications are positive.  In the following we will abbreviate ``unital completely positive'' with \emph{ucp}. A unital linear map between operator systems is completely positive if and
only if it is completely contractive, and in such a case it is necessarily
self-adjoint. Therefore by Beth's definability theorem again, the
involution and the positive cones are definable. Explicitly $x\in
M_{n}\left( X\right) $ is positive if and only if%
\begin{equation*}
\begin{bmatrix}
1_{n} & x \\ 
x & 1_{n}%
\end{bmatrix}%
\end{equation*}%
has norm at most $1$ \cite[Lemma 3.1]{paulsen_completely_2002}. Moreover the
adjoint of $x$ is the element $y$ of $X$ that minimizes the left-hand side
of Equation \ref{Equation:adjoint}. An alternative axiomatization of
operator systems in terms of the unit, the involution, and the positive
cones is suggested in \cite[Appendix B]{goldbring_kirchbergs_2014}. Since in
turn the matrix norms are definable from these items, these two
axiomatizations are equivalent.

The operator system analog $\mathrm{OMIN}_q$ of $\mathrm{MIN}_q$ has been introduced and studied in \cite{xhabli_super_2012}. It is shown there that $\mathrm{OMIN}_q$ has entirely analogous properties as $\mathrm{MIN}_q$, when one replaces operator spaces with operator systems, and (complete) contractions with unital (completely) positive maps. An \emph{$M_{q}$-system} (also called $q$-minimal operator system in \cite{xhabli_super_2012}) is a unital $M_q$-space $X$ such there is a unital
$q$-isometry $\phi :X\rightarrow C\left( K,M_{q}\right) $ such that the
image of $\phi $ is a self-adjoint subspace of $C\left( K,M_{q}\right) $.
Equivalently, $X$ is an $M_{q}$-system if and only if $X$ is a unital $M_q$-space such that $\mathrm{MIN}_{q}\left(
X\right) $ is an operator system. The above axiomatizations of operator
systems in the language of unital operator spaces and of unital $M_{q}$%
-spaces in the language of unital $M_{q}$-spaces show that $M_{q}$-systems
are axiomatizable in the language of unital $M_{q}$-spaces. Again Beth's
definability theorem shows that the all the matrix norms as well as the
positive cones and the involution are definable.

\subsection{$\mathbb{G}_q$, $\mathbb{G}_q^u$, $\mathbb{NG}$, and $\mathbb{GS}$}

It is shown in \cite[\S 3]{lupini_uniqueness_2014} that the class of
finite-dimensional $M_{q}$-spaces is a Fra\"{\i}ss\'{e} class in the sense
of \cite{ben_yaacov_fraisse_2012}. The corresponding Fra\"{\i}ss\'{e} limit $%
\mathbb{G}_{q}$ is a separable $M_{q}$-space that is characterized by the
following property: whenever $E\subset F$ are finite-dimensional $M_{q}$%
-spaces, $f:E\rightarrow \mathbb{G}_{q}$ is a linear $q$-isometry, and $%
\varepsilon >0$, then there is a linear extension $g:F\rightarrow \mathbb{G}%
_{q}$ of $f$ such that $\left\Vert g\right\Vert _{q}\left\Vert
g^{-1}\right\Vert _{q}\leq 1+\varepsilon $; see \cite[\S 3.3]%
{lupini_uniqueness_2014}.  Similarly, it is shown in \cite[\S 4.5]%
{lupini_universal_2014} that finite-dimensional $M_{q}$-systems form a Fra%
\"{\i}ss\'{e} class. The corresponding limit is a separable $M_{q}$-system $%
\mathbb{G}_{q}^{u}$ which is characterized by the same property as $\mathbb{G%
}_{q}$, where one replaces operator spaces with operator systems and linear
maps with unital linear maps; see \cite[Proposition 4.9]%
{lupini_universal_2014}.

The class of finite-dimensional $1$-exact operator spaces is shown to be a
Fra\"{\i}ss\'{e} class in \cite[\S 4]{lupini_uniqueness_2014}. The
corresponding limit is the (noncommutative) Gurarij operator space $\mathbb{%
NG}$, characterized by the property that given finite-dimensional $1$-exact
operator spaces $E\subset F$, a linear complete isometry $\phi :E\rightarrow 
\mathbb{NG}$, and $\varepsilon >0$, there exists an injective linear map $%
\widehat{\phi }:F\rightarrow \mathbb{NG}$ extending $\phi $ such that $||%
\widehat{\phi }||_{cb}{}||\widehat{\phi }^{-1}||_{cb}\leq 1+\varepsilon $.
Similarly the class of finite-dimensional $1$-exact operator systems is
shown to be a Fra\"{\i}ss\'{e} class in \cite[\S 3]{lupini_universal_2014}.
The corresponding limit is the Gurarij operator space $\mathbb{GS}$,
characterized by a similar property as $\mathbb{NG}$ where one replaces
operator spaces with operator systems and linear maps with unital\emph{\ }%
linear maps.

At this point, it is worth recording the following:

\begin{prop}\label{Proposition:noncategorical}
Both $\Th(\mathbb{NG})$ and $\Th(\mathbb{GS})$ have a continuum of nonisomorphic separable models.
\end{prop}

\begin{proof}
Every operator space admits a completely isometric embedding into $\prod_{\mathcal{U}} M_n$, the operator space ultraproduct of matrix algebras. Furthermore $M_n$ admits a completely isometric embedding into $\mathbb{NG}$ by universality. Henceforth any separable operator space embeds into a separable model of the theory of $\mathbb{NG}$. Suppose now, towards a contradiction, that $\kappa <\mathfrak{c}$ and $(Z_i)_{i<\kappa}$ enumerate all of the separable models of the theory of $\ng$ up to complete isometry.  Let $Z=\bigoplus_{i<\kappa}Z_i$.  If $X$ is any separable operator space, then $X$ embeds into some $Z_i$ and hence embeds into $Z$.  It follows that $Z$ is an operator space of density character $\kappa $ that contains all separable operator spaces. This contradicts the fact that for $n\geq 3$ the space of $n$-dimensional operator spaces has density character $\mathfrak{c}$ with respect to the completely bounded distance, which is the main result of \cite{junge_bilinear_1995} as formulated in \cite{pisier_introduction_2003}, Corollary 21.15 and subsequent remark. The assertion about $\mathbb{GS}$ can be proved in the same way, after observing that any separable operator space embeds into a separable operator system.
\end{proof}

Let $\mathrm{Aut}\left( \mathbb{NG}\right) $ be the automorphism group of $%
\mathbb{NG}$, i.e.\ the group of surjective linear complete isometries $%
\alpha :\mathbb{NG}\rightarrow \mathbb{NG}$, endowed with the topology of
pointwise convergence. The homogeneity property of $\mathbb{NG}$ implies
that every point in the unit sphere of $\mathbb{NG}$ has dense orbit under
the action of $\mathrm{Aut}\left( \mathbb{NG}\right) $. Therefore it follows
from Proposition \ref{Proposition:noncategorical} and \cite[Theorem 2.4]%
{ben_yaacov_weakly_2013} that $\mathrm{Aut}\left( \mathbb{NG}\right) $ is
not Roelcke precompact; see \cite[Definition 1.1]{ben_yaacov_weakly_2013}.

\section{The operator spaces \texorpdfstring{$\mathbb{G}_q$}{Gq} and the operator systems \texorpdfstring{$\mathbb{G}_q^u$}{Gqu}}

\subsection{The operator spaces \texorpdfstring{$\mathbb{G}_q$}{Gq}\label{Subsection:Gq}}

The following amalgamation result is proved in \cite[Lemma 3.1]%
{lupini_uniqueness_2014}; see also \cite[Lemma 2.1]{lupini_operator_2015}.

\begin{lemma}
\label{Lemma:amalgamation-Mq}If $X\subset \widehat{X}$ and $Y$ are $M_{q}$%
-spaces, and $f:X\rightarrow Y$ is a linear injective $q$-contraction such
that $\left\Vert f^{-1}\right\Vert _{q}\leq 1+\delta $, then there exists an 
$M_{q}$-space $Z$ and $q$-isometric linear maps $i:\widehat{X}\rightarrow Z$
and $j:Y\rightarrow Z$ such that $\left\Vert i_{|X}-j\circ f\right\Vert
_{q}\leq \delta $.
\end{lemma}

Arguing as in the proof of \cite[Theorem 1.1]{kubis_proof_2013}, where \cite[%
Lemma 2.1]{kubis_proof_2013} is replaced by Lemma \ref{Lemma:amalgamation-Mq}%
, shows that $\mathbb{G}_{q}$ has following homogeneity property: whenever $%
X $ is a finite-dimensional subspace of $\mathbb{G}_{q}$ and $\phi
:X\rightarrow \mathbb{GS}$ is a linear map such that $\left\Vert \phi
\right\Vert _{q}<1+\delta $ and $\left\Vert \phi ^{-1}\right\Vert
_{q}<1+\delta $, there exists a linear surjective $q$-isometry $\alpha :%
\mathbb{G}_{q}\rightarrow \mathbb{G}_{q}$ such that $\left\Vert \alpha
_{|X}-\phi \right\Vert _{q}<\delta $.

\begin{prop}
\label{Proposition:axioms_Gq} $\Th(\mathbb{G}_{q})$ is separably categorical.
\end{prop}

\begin{proof}
Suppose that $E\subset F$ are finite-dimensional $M_{q}$-spaces, where $E$
has dimension $k$ and $F$ has dimension $m>k$. Fix also a normalized basis $%
\vec{a}=\left( a_{1},\ldots ,a_{m}\right) $ of $F$ such that $\left(
a_{1},\ldots ,a_{k}\right) $ is a basis of $E$. For $1\leq n\leq m$ we let $%
X_{n}$ denote those $n$-tuples $(\alpha _{1},\ldots ,\alpha _{n})$ from $%
M_{q}$ such that

\begin{equation*}
\left\Vert \sum_{i=1}^{n}\alpha _{i}\otimes e_{i}\right\Vert =1\text{.}
\end{equation*}%
Note that $X_{n}$ is a compact subset of $M_{q}^{n}$, whence definable. We
then let $\eta _{\vec{a},n}\left( x_{1},\ldots ,x_{n}\right) $ denote the
formula%
\begin{equation*}
\sup_{\left( \alpha _{1},\ldots ,\alpha _{n}\right) \in X_{n}}\max \left\{
\left\Vert \sum_{i=1}^{n}\alpha _{i}\otimes x_{i}\right\Vert \dotminus1,1%
\dotminus\left\Vert \sum_{i=1}^{n}\alpha _{i}\otimes x_{i}\right\Vert
\right\} 
\end{equation*}%
where $a\dotminus b$ denotes the maximum of $a-b$ and $0$. For the sake of
brevity, we write $\eta _{\vec{a},n}(\vec{x})$ instead of $\eta _{\vec{a}%
,n}(x_{1},\ldots ,x_{n})$; no confusion should arise as the subscript
indicates what the free variables are. Furthermore define $\theta _{\vec{a}%
,k}\left( x_{1},\ldots ,x_{k},y_{1},\ldots ,y_{m}\right) $ to be the formula%
\begin{equation*}
\max \left\{ \eta _{\vec{a},m}(\vec{y}),\sup_{\left( \beta _{1},\ldots
,\beta _{k}\right) \in X_{k}}\left\Vert \sum_{i=1}^{k}\beta _{i}\otimes
\left( x_{i}-y_{i}\right) \right\Vert \dotminus2\eta _{\vec{a},k}(\vec{x}%
)\right\} \text{.}
\end{equation*}%
We now let $\sigma _{\vec{a},k}$ denote the sentence 
\begin{equation*}
\sup_{x_{1},\ldots ,x_{k}}\min \left\{ \frac{1}{4}\dotminus\eta _{\vec{a}%
,k}\left( \vec{x}\right) ,\inf_{y_{1},\ldots ,y_{m}}\theta _{\vec{a},k}(\vec{%
x},\vec{y})\right\} \text{.}
\end{equation*}

\noindent \textbf{Claim 1:} $\sigma _{\vec{a},k}^{\mathbb{G}_{q}}=0$.

\noindent \textbf{Proof of Claim 1:} Suppose that $b_{1},\ldots ,b_{k}$ are
elements in the unit ball of $\mathbb{G}_{q}$ such that $\eta _{\vec{a},k}(%
\vec{b})<\frac{1}{4}$. Fix $\delta \in (0,\frac{1}{4}]$ such that $\eta _{%
\vec{a},k}(\vec{b})<\delta $. Define the linear map $f:E\rightarrow \mathbb{G%
}_{q}$ by $f\left( a_{i}\right) =b_{i}$ for $i\leq k$. Observe that $%
\left\Vert f\right\Vert _{q}<1+\delta $ and $\left\Vert f^{-1}\right\Vert
_{q}<1+2\delta $. Therefore by the above mentioned homogeneity property of $%
\mathbb{G}_{q}$ there exists a linear $q$-isometry $q:F\rightarrow \mathbb{G}%
_{q}$ such that $\left\Vert g_{|E}-f\right\Vert _{q}<2\delta $. Let $%
c_{i}=g\left( a_{i}\right) $ for $1\leq i\leq m$ and observe that $\theta (%
\vec{b},\vec{c})=0$.

\noindent \textbf{Claim 2:} If $Z$ is a separable $M_{q}$-space for which $%
\sigma _{\vec{a},k}^{Z}=0$ for each $k<m$ and $\vec{a}$ as above, then $Z$
is $q$-isometric to $\mathbb{G}_{q}$.

\noindent \textbf{Proof of Claim 2:} Suppose that $f:E\rightarrow Z$ is a
linear $q$-isometry, $\dim (E)=k$, $F$ is an $m$-dimensional $M_{q}$-space
containing $E$, and $\varepsilon >0$ is given. Fix a normalized basis $\vec{a%
}=(a_{1},\ldots ,a_{m})$ of $F$ for which $a_{1},\ldots ,a_{k}$ is a basis
of $E$, and $\eta \in (0,\frac{1}{4})$ small enough. Set $b_{i}=f\left(
a_{i}\right) $ for $i\leq k$. Since $\eta _{\vec{a},k}(\vec{b})=0$ and $%
\sigma _{\vec{a},k}^{Z}=0 $, there are $c_{i}\in Z$ for $1\leq i\leq m$ such
that $\theta ^{Z} (\vec{b},\vec{c})\leq \eta $. Therefore the linear map $%
g:F\rightarrow \mathbb{G}_{q}$ defined by $g\left( a_{i}\right) =c_{i}$ for $%
1\leq i\leq m$ is such that $\left\Vert g\right\Vert _{q}\leq 1+\eta $, $%
\left\Vert g^{-1}\right\Vert _{q}\leq 1+2\eta $, and $\left\Vert
g_{|E}-f\right\Vert _{q}\leq 1+\eta $. The \textquotedblleft small
perturbation argument\textquotedblright ---see \cite[Lemma 12.3.15]%
{brown_C*-algebras_2008} and also \cite[\S 2.13]{pisier_introduction_2003}%
---allows one to perturb $g$ to a linear map that extends $f$ while only
slighlty changing the $q$-norms of $g$ and its inverse. Upon choosing $\eta $
small enough, this shows that $Z$ satisfies the approximate homogeneity
property that characterizes $\mathbb{G}_{q}$.
\end{proof}
We now give an alternate proof of the preceding theorem using the
Ryll-Nardzewski Theorem \cite[Theorem 12.10]{ben_yaacov_model_2008}. For the definition of \emph{approximately oligomorphic action}, see \cite[Definition 2.1]{ben_yaacov_weakly_2013}.

\begin{prop}
\label{alternate} Suppose that $q\in \mathbb{N}$. Then the action of $%
\mathrm{Aut}\left( \mathbb{G}_{q}\right) $ on the unit ball $\mathrm{Ball}%
\left( \mathbb{G}_{q}\right) $ of $\mathbb{G}_{q}$ is approximately
oligomorphic.
\end{prop}

\begin{proof}
Observe that the quotient space $\mathrm{Ball}\left( \mathbb{G}_{q}\right) //%
\mathrm{Aut}\left( \mathbb{G}_{q}\right) $ is isometric to $\left[ 0,1\right]
$ and hence compact. We need to show that the quotient space $\mathrm{Ball}%
\left( \mathbb{G}_{q}\right) ^{k}//\mathrm{Aut}\left( \mathbb{G}_{q}\right) $
is compact for every $k\in \mathbb{N}$. This is essentially shown in \cite[%
Proposition 3.5]{lupini_uniqueness_2014}. We denote by $\left[ a_{1},\ldots
,a_{k}\right] $ the image of the tuple $\left( a_{1},\ldots ,a_{k}\right) $
of \ $\mathrm{Ball}\left( \mathbb{G}_{q}\right) ^{k}$ in the quotient $%
\mathrm{Ball}\left( \mathbb{G}_{q}\right) ^{k}//\mathrm{Aut}\left( \mathbb{G}%
_{q}\right) $. Suppose that $\left[ a_{1}^{\left( n\right) },\ldots
,a_{k}^{\left( n\right) }\right] $ is a sequence in $\mathrm{Ball}\left( 
\mathbb{G}_{q}\right) ^{k}//\mathrm{Aut}\left( \mathbb{G}_{q}\right) $.
After passing to a subsequence we can assume that, for every $\alpha
_{1},\ldots ,\alpha _{k}\in M_{q}$ the sequence%
\begin{equation*}
\left\Vert \alpha _{1}\otimes a_{1}^{\left( n\right) }+\cdots +\alpha
_{n}\otimes a_{k}^{\left( n\right) }\right\Vert 
\end{equation*}%
converges. This implies that the convergence is uniform on the unit ball of $%
M_{q}$. Suppose that $a_{1},\ldots ,a_{k}\in \mathbb{G}_{q}$ are such that%
\begin{equation*}
\left\Vert \alpha _{1}\otimes a_{1}+\cdots +\alpha _{n}\otimes
a_{k}\right\Vert =\lim_{n}\left\Vert \alpha _{1}\otimes a_{1}^{\left(
n\right) }+\cdots +\alpha _{n}\otimes a_{k}^{\left( n\right) }\right\Vert 
\text{.}
\end{equation*}%
Then \cite[Proposition 3.4]{lupini_uniqueness_2014} shows that $\left[
a_{1},\ldots ,a_{k}\right] $ is the limit of 
\begin{equation*}
\left( \left[ a_{1}^{\left( n\right) },\ldots ,a_{k}^{\left( n\right) }%
\right] \right) _{n\in \mathbb{N}}
\end{equation*}%
in $\mathrm{Ball}\left( \mathbb{G}_{q}\right) ^{k}//\mathrm{Aut}\left( 
\mathbb{G}_{q}\right) $. This shows that every sequence has a convergent
subsequence and hence such a space is compact.
\end{proof}

It follows from Proposition \ref{Proposition:axioms_Gq} and \cite[Theorem 2.4]%
{ben_yaacov_weakly_2013} that the automorphism group $\mathrm{Aut}\left( 
\mathbb{G}_{q}\right) $, i.e.\ the group of surjective $q$-isometric linear
maps $\alpha :\mathbb{G}_{q}\rightarrow \mathbb{G}_{q}$ is Roelcke
precompact for every $q\in \mathbb{N}$. The Roelcke compactification of a
Roelcke precompact group is described model-theoretically in \cite[\S 2.2]%
{ben_yaacov_weakly_2013}.


\subsection{Quantifier-elimination}

Recall from \cite[Proposition 13.2]{ben_yaacov_model_2008} the following
test for quantifier-elimination:

\begin{fact}
Suppose that, whenever $\mathcal{M},\mathcal{N}\models T$, $\mathcal{M}_0,%
\mathcal{N}_0$ are finitely generated substructures of $\mathcal{M}$ and $%
\mathcal{N}$ respectively, $\Phi:\mathcal{M}_0\to\mathcal{N}_0$ is an
isomorphism, $\varphi(\vec x)$ is an $\la$-formula, and $\vec a\in \mathcal{M%
}_0$, we have 
\begin{equation*}
\varphi^{\mathcal{M}}(\vec a)=\varphi^{\mathcal{N}}(\Phi(\vec a)).
\end{equation*}
Then $T$ admits quantifier-elimination.
\end{fact}

\begin{prop}
$\Th(\mathbb{G}_q)$ has quantifier-elimination.
\end{prop}

\begin{proof}
This follows immediately from the above quantifier-elimination test and the
homogeneity and separable categoricity of $\mathbb{G}_q$.
\end{proof}

\subsection{The operator systems \texorpdfstring{$\mathbb{G}_q^u$}{Gqu}}

Proceeding as in Subsection \ref{Subsection:Gq}, and using the characterization of $\mathbb{G}_{q}^{u}$ given in \cite[\S 4.5]{lupini_universal_2014}, one can prove similarly as above the following facts.

\begin{prop}
\label{Proposition:axioms_Gqu}$\Th(\mathbb{G}_{q}^{u})$ is separably categorical
\end{prop}

\begin{cor}
$\Th(\mathbb{G}_{q}^{u})$ has quantifier-elimination.
\end{cor}

\begin{rmk}
One can also prove the analogue for $\mathbb{G}_{q}^{u}$ of Proposition \ref{alternate}. In this case one needs to
use results from \cite{lupini_universal_2014} and in particular \cite[Lemma
3.8]{lupini_universal_2014}.
\end{rmk}

\section{The Guarij operator system \texorpdfstring{$\mathbb{GS}$}{GS}}

In this section, we establish the model theoretic properties of the Gurarij operator system announced in the introduction.  We first need an introductory subsection on lifting ucp maps.

\subsection{Perturbation and ucp maps}

\begin{lemma}
\label{Lemma:approx-positive}Suppose that $X$ and $Y$ are operator systems and $%
\phi :X\rightarrow Y$ is a unital linear map such that $\left\Vert \phi
\right\Vert \leq 1+\delta $. If $x\in \mathrm{Ball}\left( X\right) $ is
self-adjoint, then $\left\Vert \mathrm{Im}\left( \phi \left( x\right)
\right) \right\Vert \leq \delta +\sqrt{\delta }$. If $x$ is moreover is
positive, then $\mathrm{Re}\left( \phi \left( x\right) \right) +\frac{\delta 
}{2}\geq 0$.
\end{lemma}

\begin{proof}
The first assertion follows from \cite[Lemma 3.1]{lupini_universal_2014}.
The second assertion can be proved in a similar fashion.  Indeed, we can assume,
without loss of generality, that $Y=\mathbb{C}$, whence $\phi $ is a
unital linear functional. Observe that the spectrum $\sigma \left( x\right) $
of $x$ is contained in the closed disc of center $\frac{1}{2}$ and radius $%
\frac{1}{2}$. Therefore $\sigma (x-\frac{1}{2})$ is contained in the closed
disc of center $0$ and radius $\frac{1}{2}$. Hence $\left\Vert x-\frac{1}{2}%
\right\Vert {}\leq \frac{1}{2}$. It follows that $\left\vert \phi \left(
x\right) -\frac{1}{2}\right\vert \leq \frac{1}{2}\left( 1+\delta \right) $.
\end{proof}

\begin{lemma}
\label{Lemma:perturb-unital} Suppose that $Y$ is an operator system and $%
\phi :M_{q}\rightarrow Y$ is a unital linear map such that $\left\Vert \phi
\right\Vert _{q}\leq 1+\delta $. Then there exists a ucp map $\psi
:M_{q}\rightarrow Y$ such that $\left\Vert \psi -\phi \right\Vert _{cb}\leq
3q^{2}\delta +2q^{2}\sqrt{\delta }$.
\end{lemma}

\begin{proof}
We can assume, without loss of generality, that $\phi $ is an isometry. Let $%
b$ be the image under $\phi ^{\left( q\right) }$ of the matrix $\left[ e_{ij}%
\right] \in M_{q}\left( M_{q}\right) $, where $e_{ij}\in M_{q}$ are the
matrix units. Lemma \ref{Lemma:approx-positive} shows that there exists a
positive element $a\in M_{q}\left( Y\right) $ such that $\left\Vert
b-a\right\Vert \leq \frac{3}{2}\delta +\sqrt{\delta }$. If $\psi
_{0}:M_{q}\rightarrow A$ is the linear map such that $\psi _{0}\left(
e_{ij}\right) =a_{ij}$ then $\psi _{0}$ is completely positive by Choi's
theorem \cite{paulsen_completely_2002}. Furthermore by the small
perturbation argument, we have $\left\Vert \psi _{0}-\phi \right\Vert
_{cb}\leq \frac{3}{2}q^{2}\delta +q^{2}\sqrt{\delta }$; see also \cite[%
Proposition 2.40]{goldbring_kirchbergs_2014}. Let $\tau $ be the normalized
trace on $M_{q}$ and define $\psi :M_{q}\rightarrow Y$ by%
\begin{equation*}
\psi \left( x\right) =\psi _{0}\left( x\right) +\tau \left( x\right) \left(
\psi _{0}\left( 1\right) -1\right) \text{.}
\end{equation*}%
It follows that $\psi $ is a ucp map such that $\left\Vert \psi -\phi
\right\Vert _{cb}\leq 2\left\Vert \psi _{0}-\phi \right\Vert _{cb}\leq
3q^{2}\delta +2q^{2}\sqrt{\delta }$.
\end{proof}
Observe that, in the previous lemma, if $\left\Vert \phi ^{-1}\right\Vert
_{cb}\leq \frac{1}{1-\delta }$ and $\delta \leq \frac{1}{4}$, then $%
\left\Vert \phi \right\Vert _{cb}\leq 4q^{2}\delta +2q^{2}\sqrt{\delta }\leq
5q^{2}\sqrt{\delta }$ and $\left\Vert \psi ^{-1}\right\Vert _{cb}\leq \frac{1%
}{1-4q^{2}\delta -2q^{2}\delta ^{2}}\leq \frac{1}{1-5q^{2}\sqrt{\delta }}$.

\begin{lemma}\label{Lemma:perturb-nonunital}
Suppose that $Y$ is an operator system and $\phi :M_{q}\rightarrow Y$ is a
linear map such that $\left\Vert \phi \right\Vert _{q}\leq 1+\delta \leq 2$
and $\left\Vert \phi \left( 1\right) -1\right\Vert \leq \delta $. Then there
exists a ucp map $\psi :M_{q}\rightarrow Y$ such that $\left\Vert \psi -\phi
\right\Vert _{cb}\leq 10q^{2}\sqrt{\delta }$.
\end{lemma}

\begin{proof}
Let $\phi _{0}:M_{q}\rightarrow Y$ be defined by%
\begin{equation*}
x\mapsto \phi \left( x\right) +\tau \left( x\right) \left( 1-\phi \left(
1\right) \right) \text{.}
\end{equation*}%
Then $\left\Vert \phi _{0}-\phi \right\Vert _{cb}\leq 1+\delta $ and hence $%
\left\Vert \phi _{0}\right\Vert _{cb}\leq \left\Vert \phi \right\Vert
_{cb}+\delta \leq 1+2\delta \leq 2$. Therefore, by the previous lemma, there
exists a ucp map $\psi :M_{q}\rightarrow Y$ such that $\left\Vert \psi -\phi
_{0}\right\Vert _{cb}\leq 5q^{2}\sqrt{2\delta }$ and hence $\left\Vert \psi
-\phi \right\Vert _{cb}\leq 5q^{2}\sqrt{2\delta }+\delta \leq 10q^{2}\sqrt{%
\delta }$.
\end{proof}

Suppose that $X$ is an operator system and $\left( Y_{n}\right) $ is a sequence of operator systems. If $\phi
:X\rightarrow \prod_{\mathcal{U}}Y_{n}$ is a ucp map, then a \emph{ucp lift} for $%
\phi $ is a sequence $\left( \phi _{n}\right) $ of ucp maps from $X$ to $%
Y_{n}$ such that, for every $x\in X$, $\left( \phi _{n}\left( x\right) \right) $ is a
representative sequence for $\phi \left( x\right) $;  in formulas $\phi(x)=(\phi_n(x))^{\bullet}$. Here, and in the rest of this section, all ultrafilters are assumed to be nonprincipal
ultrafilters on $\mathbb{N}$.

\begin{cor}
\label{Corollary:lift-Mq}If $\phi :M_{q}\rightarrow \prod_{\mathcal{U}}Y_{n}$
is a ucp map, then $\phi$ has a ucp lift.
\end{cor}

Observe that, by Smith's Lemma \cite[Theorem 2.10]{smith_completely_1983}, if $\phi $ is a complete order embedding and $%
\left( \phi _{k}\right) $ is a ucp lift for $\phi $, then $\lim_{k\to \mathcal{U}}\left\Vert \phi
_{k}^{-1}\right\Vert _{cb}= 1$.

\begin{cor}
\label{Corollary:lift-GS}If $\phi :\mathbb{GS}\rightarrow \prod_{\mathcal{U}%
}Y_{n}$ is a ucp map, then $\phi$ has a ucp lift.
\end{cor}

\begin{proof}
It is observed in \cite[\S 4.3]{lupini_universal_2014} that $\mathbb{GS}$ can be
realized as the direct limit of a sequence of full matrix algebras with
complete order embeddings as connecting maps. Furthermore, it is a particular case of \cite[Theorem 3.3]{lupini_operator_2015} that, for every $n\in \mathbb{N}$, there exist a
complete order embedding $\eta $ of $M_{n}$ into $\mathbb{GS}$ and a ucp
projection of $\mathbb{GS}$ onto the range of $\eta $. It follows from these
facts and the homogeneity property of $\mathbb{GS}$ given by \cite[Theorem 4.4]
{lupini_universal_2014} that one can find a sequence $\left( X_{k}\right) $
of subsystems of $\mathbb{GS}$ such that, for every $k\in \mathbb{N}$, the following conditions hold:

\begin{enumerate}
\item $\bigcup_{n\geq k}X_{n}$ is dense in $\mathbb{GS}$;

\item $X_{k}$ is completely order isomorphic to a full matrix algebra;

\item $X_{k}$ is the range of a ucp projection $P_{k}$ of $\mathbb{GS}$; and

\item every element of $X_{k}$ is at distance at most $2^{-k}$ from some
element of $X_{k+1}$.
\end{enumerate}

For $k\in \mathbb{N}$, $\phi _{|X_{k}}$ is a ucp map and
hence, by Corollary \ref{Corollary:lift-Mq}, admits a ucp lift $\left(
\phi _{k,n}\right) $. Finally, for $n\in \mathbb{N}$, define $\phi _{n}:=\phi _{n,n}\circ P_{n}$.  It remains to observe that $(\phi_n)$ is a ucp lift for $\phi $.
\end{proof}

\begin{cor}\label{Corollary:elementary-embedding}
If $\alpha :\mathbb{GS}\rightarrow \mathbb{GS}^{\mathcal{U}}$ is a complete
order embedding, then $\alpha$ has a ucp lift $\left( \alpha _{n}\right) $ made of
complete order automorphisms $\alpha _{n}:\mathbb{GS}\rightarrow \mathbb{GS}$%
. In particular, $\alpha $ is an elementary embedding.
\end{cor}

\begin{proof}
Let $\left( X_{k}\right) $ be an increasing sequence of subsystems of $%
\mathbb{GS}$ with dense union such that $X_{k}$ is completely order
isomorphic to $M_{n_{k}}$. Let $\left( \phi _{n}\right) $ be a ucp lift of $%
\alpha $. Then, as mentioned above, for any $k\in \mathbb{N}$, we have $\lim_{n\to \mathcal{U}}\left\Vert (\phi _{n})_{|X_{k}}^{-1}\right\Vert
_{cb}=1$.
Therefore, by the homogeneity property of $\mathbb{GS}$, one can find complete
order automorphisms $\alpha _{n}:\mathbb{GS}\rightarrow \mathbb{GS}$ such
that $\lim_{n\to \mathcal{U}}\left\Vert \left( \alpha _{n}-\phi _{n}\right) _{|X_{k}}\right\Vert=0$ for any $k\in \mathbb{N}$. Therefore $\left( \alpha
_{n}\right) $ is also ucp lift of $\alpha $.
\end{proof}

\begin{nrmk}
Strongly self-absorbing C$^*$ algebras also enjoy the property that any embedding of them into their ultrapower is elementary.  However, they have this property for a completely different reason, namely that any such embedding is unitarily conjugate to the diagonal embedding.
\end{nrmk}

\subsection{$\mathbb{GS}$ is the prime model of its theory}

The main goal of this subsection is to prove Theorem \ref{Theorem:prime}, asserting that $\gs$ is the prime model of its theory. 

Suppose that $q\leq n$ are positive integers. Let $\vec{a}=(a_{0},\ldots
,a_{n^{2}-1})$ be a basis for $M_{n}$ such that $a_{0}=1$ and $\mathrm{span}%
\{a_{0},\ldots ,a_{q^{2}-1}\}$ is a subsystem of $M_{n}$ completely order
isomorphic to $M_{q}$. Let  $\eta _{\vec{a},q}(x_{1},\ldots
,x_{q^{2}-1})$ be the formula%
\begin{equation*}
\sup_{(\alpha _{0},\ldots ,\alpha _{q^{2}-1})\in X_{\vec{a},q}}\max \left\{
\left\Vert \sum_{i=0}^{q^{2}-1}\alpha _{i}\otimes x_{i}\right\Vert \dotminus 1,1\dotminus\left\Vert \sum_{i=0}^{q^{2}-1}\alpha
_{i}\otimes x_{i}\right\Vert \right\} 
\end{equation*}%
where $x_{0}=1$ and $X_{\vec{a},q}$ is the (compact) set of tuples $\alpha
_{0},\ldots ,\alpha _{q^{2}-1}\in M_{q}$ such that%
\begin{equation*}
\left\Vert \sum_{i=0}^{q^{2}-1}\alpha _{i}\otimes a_{i}\right\Vert = 1%
\text{.}
\end{equation*}%
Let $\theta _{\vec{a},q}(x_{1},\ldots ,x_{q^{2}-1},y_{1},\ldots ,y_{n^{2}-1})
$ be the formula%
\begin{equation*}
\max \left\{ \eta _{\vec{a},n}\left( \vec{y}\right) ,\sup_{(\alpha
_{0},\ldots ,\alpha _{q^{2}-1})\in X_{\vec{a},q}}\left\Vert
\sum_{i=0}^{q^{2}-1}\alpha _{i}\otimes \left( x_{i}-y_{i}\right) \right\Vert
^{2}\dotminus25q^{4}\eta _{\vec{a},q}(\vec{x})\right\} 
\end{equation*}%
where $y_{0}=1$. Finally let $\sigma _{\vec{a},q}$ be the sentence%
\begin{equation*}
\sup_{\vec{x}}\min \left\{ \frac{1}{4}\dotminus\eta _{\vec{a}%
,k}\left( \vec{x}\right) ,\inf_{y_{1},\ldots ,y_{n^{2}}-1}\theta _{\vec{a}%
,q}\left( \vec{x},\vec{y}\right) \right\} \text{.}
\end{equation*}

\begin{lemma}
Suppose that $X$ is an operator system. Then $\sigma _{\vec{a},q}^X=0$
if and only if, for any $\delta\in (0,\frac{1}{4}]$, any $\varepsilon>0$, and any unital linear map $\phi :\mathrm{span}\{a_{0},\ldots
,a_{q^{2}-1}\}\rightarrow X$ such that $\left\Vert \phi \right\Vert
_{cb}<1+\delta $ and $\left\Vert \phi ^{-1}\right\Vert _{cb}<\frac{1}{%
1-\delta }$, there
exists a ucp map $\psi :M_{n}\rightarrow X$ such that $\left\Vert \psi
^{-1}\right\Vert _{cb}<1+\varepsilon $ and $\left\Vert \psi -\phi
\right\Vert _{q}<5q^{2}\sqrt{\delta }$.
\end{lemma}

\begin{proof}
One implication is obvious and the other one follows from Lemma \ref%
{Lemma:perturb-unital}.
\end{proof}

\begin{nrmk}
It follows form the homogeneity property of $\mathbb{GS}$ that $
\sigma _{\vec{a},q}^{\mathbb{GS}}=0.$
\end{nrmk}

\begin{thm}\label{Theorem:prime}
$\gs$ is the prime model of its theory.
\end{thm}

\begin{proof}
In light of the fact that by Corollary \ref{Corollary:elementary-embedding} every embedding of $\gs$ into its ultrapower is
elementary, it suffices to show that $\gs$ embeds into any model of its
theory. Towards this end, fix an operator system $Y$ that is elementarily
equivalent to $\gs$. Let $\left( X_{k}\right) $ be a sequence of subsystems
of $\mathbb{GS}$ with dense union such that $X_{k}$ is completely order
isomorphic to $M_{n_{k}}$. Fix also a sequence $\left( \varepsilon
_{k}\right) $ of positive real numbers such that $\sum_{k}\varepsilon
_{k}<+\infty $. Using the fact that $Y$ is a model of the theory of $\mathbb{GS}$ and the lemma above, one can define by recursion on $k$ ucp maps $%
\psi _{k}:X_{k}\rightarrow Y$ such that:

\begin{enumerate}
\item $\left\Vert \psi _{k}^{-1}\right\Vert _{cb}<1+\varepsilon _{k}$, and

\item $\left\Vert (\psi _{k+1})_{|X_{k}}-\psi _{k}\right\Vert <\varepsilon
_{k}.$
\end{enumerate}

Define $\psi :\mathbb{GS}\rightarrow X$ by setting $\psi
\left( x\right) =\lim_{k\rightarrow +\infty }\psi _{k}\left( x\right) $ for $%
x\in \bigcup_{k}M_{k}$ and extending by continuity. Observe that $\psi $ is well defined by (2) and it is a complete order embedding by (1).
\end{proof}


\begin{nrmk}
Theorem \ref{Theorem:prime} should be compared with \cite[Proposition 5.1]{goldbring_kirchbergs_2014}, asserting that $\mathcal{O}_2$ is the prime model of its theory.
\end{nrmk}

\subsection{$\mathbb{GS}$ does not have quantifier elimination}

In this section, we prove that $\gs$ does not have quantifier elimination.  In fact, we offer two (very different) proofs.  The first proof relies on the following fact.

\begin{fact}[\cite{goldbring_omitting_2015}]\label{goldsinc}
There does \emph{not} exist a family $\Gamma_{m,n}(\vec{x}_m)$ of definable predicates in the language of operator systems (taking only nonnegative values) for which an operator system $E$ is $1$-exact if and only if, for every $m$ and every $\vec a\in E^{\vec{x}_m}$, we have $\inf_n\Gamma_{m,n}^E(\vec a)=0$.
\end{fact}

\begin{thm}\label{Theorem:noQE}
$\Th(\gs)$ does not have quantifier-elimination.
\end{thm}
\begin{proof}


Suppose, towards a contradiction, that $\gs$ has quantifier elimination.  For each $m$, let $(b_{m,n})_n$ denote a countable dense subset of (the unit ball of) $\gs^m$.  Let $p_{m,n}:=\tp^{\mathbb{GS}}(b_{m,n})$.  Since $\gs$ is the prime model (whence atomic), each $p_{m,n}$ is isolated, so the predicate $d(\cdot,p_{m,n})$ is a definable predicate, meaning that there are formula $\varphi_{m,n}^k(\vec x_m)$ such that, for all $\vec a$ (in a monster model of $\Th(\gs)$), $d(\vec a,p_{m,n})=\lim_k \varphi_{m,n}^k(\vec a)$ uniformly.  Since we are assuming \emph{a contrario} that $\Th(\gs)$ has quantifier elimination, we may further suppose that the formulae $\varphi_{m,n}^k$ are quantifier-free.  Since any operator system embeds into a model of $\Th(\mathbb{GS})$, we can thus consider the definable (relative to the elementary class of operator systems) predicates $\Gamma_{m,n}(\vec x_m):=\lim_k \varphi_{m,n}^k(\vec x_m)$; see \cite[Section 3.2]{ben_yaacov_continuous_2010}.

We obtain the desired contradiction by showing that an operator system $E$ is 1-exact if and only if $\inf_n \Gamma_{m,n}^E(\vec a)=0$ for all $m$ and all $\vec a\in E^m$.

First suppose that $\inf_n\Gamma_{m,n}^E(\vec a)=0$ for all $m$ and all $\vec a\in E^m$.  In order to show that $E$ is 1-exact, it suffices to show that all of its finite-dimensional subsystems are 1-exact.  Thus, without loss of generality, we may assume that $E$ is the operator system generated by $\vec a$ for some linearly independent tuple $\vec a$.  Fix $M\models \Th(\gs)$ containing $E$.  Since $\inf_n \Gamma_{m,n}^M(\vec a)=\inf_n\Gamma_{m,n}^E(\vec a)=0$ (as $\Gamma$ is a quantifier-free definable predicate), we have that $\tp^M(\vec a)$ is in the metric closure of the isolated types, whence is itself isolated.  Since isolated types are realized in all models, there is $\vec b\in \gs^m$ such that $\tp^M(\vec a)=\tp^{\gs}(\vec b)$.  It follows that the map $a_i\mapsto b_i$ is a complete isometry, whence $E$ is 1-exact.

Conversely, suppose that $E$ is 1-exact.  Fix $\vec a\in E^m$.  We must show that $\inf_n\Gamma_{m,n}^E(\vec a)=0$.  Without loss of generality, we may assume that $E$ is separable, whence we may further assume that $E$ is a subsystem of $\gs$.  Since $\Gamma$ is a  quantifier-free definable predicate, we have that $\inf_n\Gamma_{m,n}^E(\vec a)=\inf_n\Gamma_{m,n}^{\gs}(\vec a)=0$ as $(b_{m,n})$ is dense in $\gs^m$. 
\end{proof}

\begin{cor}
There is $q\in \mathbb{N}$ such that $\gs$ is not unitally $q$-isometric to $\mathbb{G}_q^u$.
\end{cor}

\begin{proof}
Since any formula in the language of operator systems is a formula in the language of $M_q$-systems for some $q$, if $\gs$ were unitally $q$-isometric to $\mathbb{G}_q^u$ for every $q$, then quantifier-elimination for $\mathbb{G}_q^u$ would imply quantifier elimination for $\gs$.
\end{proof}

The second proof that $\gs$ does not have quantifier elimination is analogous to the proof given in \cite{eagle_quantifier_2015} that shows that $\O_2$ does not have quantifier-elimination.

The following lemma can be proved in a manner similar to the proof of \cite[Proposition 1.16]%
{eagle_quantifier_2015}; see also \cite[Proposition 13.6]%
{ben_yaacov_model_2008}.

\begin{lemma}
\label{Lemma:QE}Suppose that $X$ is an operator system.  Then the following
statements are equivalent:

\begin{enumerate}
\item $X$ has quantifier elimination;

\item If $Y$ is a separable operator system elementarily equivalent to 
$X$ and $Y_{0}$ is a subsystem of $Y$, then any complete order embedding of $%
Y_{0}$ into an ultrapower $X^{\mathcal{U}}$ of $X$ can be extended to a
complete order embedding of $Y$ into $X^{\mathcal{U}}$.
\end{enumerate}
\end{lemma}

A unital C*-algebra $A$ is \emph{quasidiagonal }if, for every finite subset $%
F $ of $A$ and every $\varepsilon >0$, there exists $n\in \mathbb{N}$ and a
ucp map $\phi :A\rightarrow M_{n}$ such that 
\begin{equation*}
\left\Vert \phi \left( ab\right) -\phi \left( a\right) \phi \left( b\right)
\right\Vert <\varepsilon \quad \text{and}\quad \left\Vert \phi \left(
a\right) \right\Vert >\left\Vert a\right\Vert -\varepsilon \text{.\label%
{Equation:qd}}
\end{equation*}%
Equivalently $A$ is quasidiagonal if there exits a unital injective
*-homomorphism $\phi :A\rightarrow M$ that admits a ucp lift, where $M$ is
the ultraproduct $\prod_{\mathcal{U}}M_{n}$.

We now can give the second proof that the theory of $\mathbb{GS}$ does not have quantifier elimination.  Let $\mathbb{F}_{2}$ be the free group on two generators. By a result of
Rosenberg from \cite{hadwin_strongly_1987}, the reduced C*-algebra $%
C_{r}^{\ast }\left( \mathbb{F}_{2}\right) $ is not quasidiagonal. However, by
a result of Haagerup and Thorbjorsen from \cite{haagerup_new_2005}, there exists
an injective unital *-homomorphism $\phi $ from $C_{r}^{\ast }\left( \mathbb{%
F}_{2}\right) $ to the ultraproduct $M:=\prod_{\mathcal{U}}M_{n}$. As
mentioned above, it is shown in \cite{lupini_operator_2015} that, for every $%
n\in \mathbb{N}$, there exists a complete order embedding $\eta
_{n}:M_{n}\rightarrow \mathbb{GS}$ and a ucp projection $P_{n}$ from $%
\mathbb{GS}$ onto the range of $\eta _{n}$. Let $\eta=(\eta_n)^{\bullet} :M\rightarrow \mathbb{%
GS}^{\mathcal{U}}$. Then $\eta \circ \phi :C_{r}^{\ast }\left( \mathbb{F}%
_{2}\right) \rightarrow \mathbb{GS}$ is a complete order embedding. Since $%
C_{r}^{\ast }\left( \mathbb{F}_{2}\right) $ a unital exact C*-algebra, one
can regard $C_{r}^{\ast }\left( \mathbb{F}_{2}\right) $ as a subsystem of $%
\mathbb{GS}$ by universality. We claim that $\eta \circ \phi $ has no
extension to a linear complete isometry $\psi :\mathbb{GS}\rightarrow 
\mathbb{GS}^{\mathcal{U}}$. Indeed, if $\psi $ is such an extension, then by
Corollary \ref{Corollary:lift-GS}, $\psi $ has a ucp lift $\left( \psi
_{n}\right) $, whence the sequence $(P_{n}\circ \left( \psi _{n}\right)
_{|C_{r}^{\ast }\left( \mathbb{F}_{2}\right) })$ is a ucp lift for $\phi $,
contradicting the fact that $C_{r}^{\ast }\left( \mathbb{F}_{2}\right) $ is
not quasidiagonal. One can then conclude that $\mathbb{GS}$ does not have
quantifier elimination by applying Lemma \ref{Lemma:QE}. 

\begin{nrmk}
The first proof that $\Th(\mathbb{GS})$ does not have quantifier elimination 
relies on the work of Junge and Pisier  in \cite{junge_bilinear_1995}, which (essentially) shows that the set of $n$-dimensional 1-exact operator systems is not a Polish space in the weak topology. Instead, the second proof uses in an essential way the aforementioned deep result of Haagerup and Thorbjorsen from \cite{haagerup_new_2005}.
\end{nrmk}

\subsection{Existentially closed operator systems}

Recall that an operator system $X$ is \emph{nuclear} if there exist nets $\rho
_{\alpha }:X\rightarrow M_{n_{\alpha }}$ and $\gamma _{\alpha }:M_{n_{\alpha
}}\rightarrow X$ of ucp maps such that $\gamma _{\alpha }\circ \rho _{\alpha
}$ converge pointwise to the identity map of $X$; see \cite[Theorem 3.1]%
{han_approximation_2011}. We also recall that an operator system $X$ is \emph{existentially
closed }if, whenever $Y$ is an operator system containing $X$, $\varphi (%
\vec{x},y)$ is a quantifier-free formula in the language of operator
systems, and $\vec{a}$ is a tuple of elements of $X$, one has that%
\begin{equation}
\inf_{b\in \mathrm{Ball}\left( Y\right) }\varphi (\vec{a},b)=\inf_{b\in 
\mathrm{Ball}\left( X\right) }\varphi (\vec{a},b)\label{Equation:ec}\text{.}
\end{equation}%
The operator system $X$ is \emph{positively existentially closed }if one merely assumes that Equation \eqref{Equation:ec} holds for quantifier-free formulas $\phi $ constructed using only nondecreasing functions as
connectives. Similar definitions can be given for operator spaces.

Note that the formulae $\theta_{\vec a,q}$ from the proof of Theorem \ref{Theorem:prime} are really quantifier-free definable predicates (as the supremum over a compact set is a limit of maxima over finer finite nets).  Furthermore the category of operator systems admits \emph{approximate pushouts} as in \cite[Lemma 3.1]{lupini_operator_2015}. This can be seen as in \cite[Lemma 3.1]{lupini_operator_2015} by replacing $M_q$ with $B(H)$. It follows that the proof of Theorem \ref{Theorem:prime}, whith the extra ingredient of approximate pushouts, also shows:

\begin{thm}
$\gs$ embeds into any existentially closed operator system.
\end{thm}

The next two results explore what happens when we combine 1-exactness and existential closedness.

\begin{prop}
\label{Proposition:pec}Suppose that $X$ is a $1$-exact operator system. Then 
$X$ is nuclear if and only if it is positively existentially closed.
\end{prop}

\begin{proof}
Suppose that $X$ is a nuclear operator system. Let $X$ be a
subsystem of $Z$. Let $\varphi \left( \vec{x},y\right) $ be a positive
quantifier-free formula in the language of operator systems and $\vec{a}$ a
tuple in $X$. Suppose that $r>0$ and $b\in Z$ is such that $\varphi \left( 
\vec{a},b\right) <r$. Fix $\varepsilon >0$ such that $\varphi \left( \vec{a}%
,b\right) +\varepsilon <r$ and $\delta >0$ small enough. Consider the
inclusion maps $\phi :\left\langle \vec{a}\right\rangle \rightarrow
\left\langle \vec{a},b\right\rangle $ and $f:\left\langle \vec{a}%
\right\rangle \rightarrow X$. By the implication\ (3)$\Rightarrow $(1) of 
\cite[Lemma 3.3]{lupini_operator_2015}, there exists a ucp map $%
g:\left\langle \vec{a},b\right\rangle \rightarrow X$ such that $\left\Vert
g\circ \phi -f\right\Vert <\delta $. Thus, for $\delta $ small enough, we have%
\begin{equation*}
\varphi \left( \vec{a},g(b)\right) \leq \varphi \left( g\left( \phi(\vec{a%
})\right) ,g\left( b\right) \right) +\varepsilon= \varphi\left( \vec{a}%
,b\right)+\varepsilon <r\text{.}
\end{equation*}%
Suppose now that $X$ is a positively existentially closed operator system.
We want to show that $X$ is nuclear. Suppose that $q\in \mathbb{N}$, $E$ is
a subsystem of $M_{q}$, and $f:E\rightarrow X$ is a ucp map. Fix $%
\varepsilon >0$. We want to show that there exists a ucp map $%
g:M_{q}\rightarrow X$ such that $\left\Vert g_{|E}-f\right\Vert <\varepsilon 
$. This will imply that $X$ is nuclear by the implication (2)$\Rightarrow $%
(3) in \cite[Lemma 3.3]{lupini_operator_2015}. Let $\delta >0$ be small
enough. Arguing as in the proof of \cite[Lemma 3.1]{lupini_operator_2015}
(where one replaces $M_{q}$ with $B\left( H\right) $) one can show that there
exist a ucp map $i:M_{q}\rightarrow B\left( H\right) $ and a complete order
embedding $j:X\rightarrow B\left( H\right) $ such that $\left\Vert
i_{|E}-j\circ f\right\Vert <\delta $. Let $\vec{c}=(c_{0},\ldots
,c_{q^{2}-1})$ be a normalized basis of $M_{q}$ such that $\left(
c_{0},\ldots ,c_{k-1}\right) $ is a basis of $E$ and let $\vec{a}=\left(
f\left( c_{0}\right) ,\ldots ,f\left( c_{k-1}\right) \right) $ and $\vec{b}%
=(g\left( c_{0}\right) ,\ldots ,g(c_{q^{2}-1}))$. Consider the formula $\eta
(y_{1},\ldots ,y_{q^{2}-1})$ defined by%
\begin{equation*}
\sup_{(\alpha _{0},\ldots ,\alpha _{q^{2}-1})\in X_{\vec{a},q}}\left\Vert
\sum_{i=0}^{q^{2}-1}\alpha _{i}\otimes y_{i}\right\Vert \dotminus1,
\end{equation*}%
where $X_{\vec{a},q}$ is the (compact) set of tuples $(\alpha _{0},\ldots
,\alpha _{q^{2}-1})$ in $M_{q}$ such that%
\begin{equation*}
\left\Vert \sum_{i=0}^{q^{2}-1}\alpha _{i}\otimes a_{i}\right\Vert \leq 1%
\text{.}
\end{equation*}%
Consider then the formula $\theta (x_{1},\ldots ,x_{k-1},y_{1},\ldots
,y_{q^{2}-1})$ defined by%
\begin{equation*}
\max \left\{ \eta (y_{1},\ldots ,y_{q^{2}-1}),\sup_{(\alpha _{0},\ldots
,\alpha _{k-1})\in X_{\vec{a},q}}\left\Vert \sum_{i-0}^{k-1}\alpha
_{i}\otimes \left( x_{i}-y_{i}\right) \right\Vert \right\} 
\end{equation*}%
where $x_{0}=y_{0}=1$ and $X_{\vec{a},q}$ is defined as above. Then $\theta (%
\vec{a},\vec{b})=0$; since $X$ is positively existentially
closed, there exists a tuple $\vec{d}$ in $X$ such that $\theta (\vec{a},\vec{%
d})<\delta $. Let $\psi :M_{q}\rightarrow X$ be the map sending $a_{0}
$ to $1$ and $a_{i}$ to $d_{i}$ for $i<q^{2}$; then $\psi$ is a unital linear map such
that $\left\Vert \psi \right\Vert _{q}<1+\delta $ and $\left\Vert \psi
_{|E}-f\right\Vert _{q}<\delta $. By Lemma \ref{Lemma:perturb-unital}, one
can find a ucp map $g:M_{q}\rightarrow X$ such that $\left\Vert g-\psi
\right\Vert \leq 5q^{2}\sqrt{\delta }$. This concludes the proof that $X$ is
nuclear by \cite[Lemma 3.2]{lupini_operator_2015}.
\end{proof}

We thus see that being positively existentially closed characterizes nuclear operator
systems among the $1$-exact operator systems. We now see that being existentially
closed characterizes the Gurarij operator system $\mathbb{GS}$ among the
separable $1$-exact operator systems.

\begin{thm}
\label{Theorem:ec}Suppose that $X$ is a separable $1$-exact operator system.
Then $X$ is existentially closed if and only if $X$ is completely order
isomorphic to $\mathbb{GS}$.
\end{thm}

\begin{proof}
One can prove that if $X$ is an existentially closed operator system, then $X
$ is completely order isomorphic to $\mathbb{GS}$ arguing as in the proof of
Proposition \ref{Proposition:pec}, where one uses \cite[Lemma 3.1]%
{lupini_operator_2015} and replaces the formula 
\begin{equation*}
\sup_{(\alpha _{0},\ldots ,\alpha _{q^{2}-1})\in X_{\vec{a},q}}\left\Vert
\sum_{i=0}^{q^{2}-1}\alpha _{i}\otimes y_{i}\right\Vert \dotminus1
\end{equation*}%
with the formula%
\begin{equation*}
\sup_{(\alpha _{0},\ldots ,\alpha _{q^{2}-1})\in X_{\vec{a},q}}\max \left\{
\left\Vert \sum_{i=0}^{q^{2}-1}\alpha _{i}\otimes y_{i}\right\Vert \dotminus%
1,1\dotminus\left\Vert \sum_{i=0}^{q^{2}-1}\alpha _{i}\otimes
y_{i}\right\Vert \right\} \text{.}
\end{equation*}%
We now show that $\mathbb{GS}$ is existentially closed. Suppose that $X$
is an operator system containing $\mathbb{GS}$. Let $\vec{a}$ be a tuple of
elements of $\mathbb{GS}$ and $\varphi \left( \vec{a},x\right) $ be a
quantifier-free formula in the language of operator systems. Since $\mathbb{%
GS}$ is the direct limit of a sequence of full matrix algebras with unital
completely isometric connective maps, without loss of generality we can
assume that the operator system $\left\langle \vec{a}\right\rangle $
generated by the tuple $\vec{a}$ in $\mathbb{GS}$ admits a complete order
embedding into $M_{q}$ for some $q\in \mathbb{N}$. We may also assume that $q$ is chosen large enough so that
only appear matrix norms up to order $q$ appear in $\varphi$. Suppose that $r\in \mathbb{%
N}$ and $\varphi \left( \vec{a},b\right) <r$ for some $b\in \mathrm{Ball}\left(
X\right) $. Consider the complete order embeddings $%
\left\langle \vec{a}\right\rangle \subset \mathrm{MIN}_{q}\left(
\left\langle \vec{a},b\right\rangle \right) $ (inclusion map) and $%
f:\left\langle \vec{a}\right\rangle \subset \mathbb{GS}$. Fix $\varepsilon >0
$ such that $\varphi \left( \vec{a},b\right) +\varepsilon <r$ and fix $\delta >0$ sufficiently small.
By the homogeneity property of $\mathbb{GS}$, there exists a complete
isometry $g:\mathrm{MIN}_{q}\left( \left\langle \vec{a},b\right\rangle
\right) \rightarrow \mathbb{GS}$ such that $\left\Vert g_{|\left\langle \vec{%
a}\right\rangle }-f\right\Vert <\delta $. It is clear that upon choosing $%
\delta $ small enough one can ensure that 
\begin{equation*}
\varphi \left( \vec{a},g\left( b\right) \right) \leq \varphi \left( g\left( 
\vec{a}\right) ,g\left( b\right) \right) +\varepsilon =\varphi \left( \vec{a}%
,b\right) +\varepsilon <r\text{.}
\end{equation*}%
This concludes the proof that $\mathbb{GS}$ is existentially closed.
\end{proof}

\begin{nrmk}

Once we have established that $\mathbb{GS}$ is existentially closed, we
obtain another proof of the fact that separable nuclear operator systems are
positively existentially closed. Indeed, suppose that $X$ is separable and
nuclear.\ By \cite[Theorem 3.3]{lupini_operator_2015} we may assume that $X$
is a subsystem of $\mathbb{GS}$, and there exists a ucp projection $\phi $
of $\mathbb{GS}\ $onto $X$. It suffices to show that $X$ is positively
existentially closed in $\mathbb{GS}$, that is, whenever $\varphi (\vec{a},x)
$ is a positive quantifier-free formula and $\vec{a}$ is a tuple from $X$,
then we have that 
\begin{equation*}
\inf_{b\in \mathrm{Ball}\left( \mathbb{GS}\right) }\varphi (\vec{a}%
,b)=\inf_{b\in \mathrm{Ball}\left( X\right) }\varphi (\vec{a},b)\text{.}
\end{equation*}%
However if $b\in \mathrm{Ball}(\mathbb{GS})$ then $\varphi (\vec{a},\phi
(b))\leq \varphi (\vec{a},b)$, whence the desired result follows.
\end{nrmk}

Theorem 3.1 of \cite{eagle_quantifier_2015} shows that the theory of C*-algebras does not have a model companion.  We now have the same conclusion for the theory of operator systems:

\begin{cor}
\label{Corollary:no-companion}The theory of operator systems does not have a
model companion.
\end{cor}

\begin{proof}
If the theory of operator systems had a model companion, then it would be a
model-completion as the class of operator systems satisfies the amalgamation
property; see \cite[\S 2]{kerr_gromov-hausdorff_2009} and also the proof of Lemma 3.1 in \cite{lupini_operator_2015} where one replaces $M_q$ with $B(H)$. Since $\gs$ would be a model of the model-completion, we would
conclude that $\gs$ has quantifier-elimination, contradicting Theorem \ref%
{Theorem:noQE}.
\end{proof}

\begin{nrmk}
In a similar manner, one can show that $\mathbb{NG}$ is an existentially closed operator space. However, since we do not know whether or not $\mathbb{NG}$ has quantifier-elimination, it is still open as to whether or not the theory of operator spaces has a model companion.  Likewise, we do not know if $\mathbb{NG}$ is the only separable $1$-exact existentially closed operator space. 
\end{nrmk}

\begin{nrmk}
Theorem \ref{Theorem:ec} should be compared with \cite[Proposition 2.18]%
{goldbring_kirchbergs_2014} asserting that the Cuntz algebra $\mathcal{O}_{2}
$ is the only possible existentially closed exact C*-algebra.  By \cite[Theorem 3.3]%
{goldbring_kirchbergs_2014}, the assertion that $%
\mathcal{O}_{2}$ is indeed existentially closed is equivalent to a positive solution to the
Kirchberg embedding problem \cite[\S 3.1]{goldbring_kirchbergs_2014}.
\end{nrmk}

\subsection{Nuclear models of the theory of $\mathbb{GS}$}

It follows from \cite[Corollary 2.9]{goldbring_kirchbergs_2014} that $%
\mathcal{O}_{2}$ is the only nuclear model of its theory. A similar
assertion holds for the Gurarij operator system $\mathbb{GS}$.

\begin{lemma}
\label{Lemma:elementary-norm}Suppose that $q\in \mathbb{N}$, $E\subset M_{q}$
is a subsystem of dimension $k$, and $\vec{a}=\left( a_{0},\ldots
,a_{k-1}\right) $ is a normalized basis of $E$ with $a_{0}=1$. Then there
exists a sequence of formulas $\theta _{m}\left( x_{1},\ldots
,x_{k-1}\right) $ in the language of operator systems such that the
following holds: if $X$ is a nuclear operator system and $\vec{b}=\left(
b_{1},\ldots ,b_{k-1}\right) $ is a tuple in the unit ball of $X$, then for
every $m\in \mathbb{N}$, if $\theta _{m}^{X}(\vec{b})<\frac{1}{4}$, then the
unital linear map $\phi :E\rightarrow X$ such that $\phi \left( a_{i}\right)
=b_{i}$ for $1\leq i\leq k-1$ is invertible and 
\begin{equation*}
\max \left\{ \left\Vert \phi \right\Vert _{cb},\left\Vert \phi
^{-1}\right\Vert _{cb}\right\} -1\leq \theta _{m}^{X}(\vec{b})\text{.}
\end{equation*}%
Furthermore, if $\phi :E\rightarrow X$ is an invertible unital linear map
and $\vec{b}=\phi \left( \vec{a}\right) $, then%
\begin{equation*}
\inf_{m}\theta _{m}^{X}(\vec{b})\leq \max \left\{ \left\Vert \phi
\right\Vert _{cb},\left\Vert \phi ^{-1}\right\Vert _{cb}\right\} -1\text{.}
\end{equation*}
\end{lemma}

\begin{proof}
Suppose that $n\geq q$ and $\vec{c}$ is a $k$-tuple in $M_{n}$ with $c_{0}=1$%
. Denote by $\left( a_{0}^{\prime },\ldots ,a_{k-1}^{\prime }\right) $ the
dual basis of $\left( a_{0},\ldots ,a_{k-1}\right) $. Let $L\geq 1$ be such
that $\left\Vert a_{i}^{\prime }\right\Vert \leq L$ for $i<k$. Denote by $%
e_{ij}$ for $i,j\leq n$ the canonical matrix units of $M_{n}$, and let $%
c_{\ell }=\sum_{ij}\lambda _{ij}^{\left( \ell \right) }e_{ij}\in M_{n}$ for $%
0\leq \ell \leq k-1$. Consider the formula $\mu _{\vec{c}}\left(
x_{1},\ldots ,x_{k-1}\right) $ defined by%
\begin{equation*}
\inf_{\lbrack y_{ij}]\in C_{n}}\sup_{(\alpha _{0},\ldots ,\alpha _{k-1})\in
X_{\vec{c},n}}\left\Vert \sum_{\ell =0}^{k-1}\alpha _{\ell }\otimes \left(
x_{\ell }-\sum_{1\leq i,j\leq n}\lambda _{ij}^{\left( \ell \right)
}y_{ij}\right) \right\Vert 
\end{equation*}%
where $x_{0}=1$, $X_{\vec{c},n}$ is the (compact) set of $k$-tuples $\left(
\alpha _{0},\ldots ,\alpha _{k-1}\right) $ in $M_{n}$ such that%
\begin{equation*}
\left\Vert \sum_{i=0}^{k-1}\alpha _{i}\otimes c_{i}\right\Vert =1\text{,}
\end{equation*}%
and $C_{n}$ is the (definable) set of positive elements of norm at most $1$
of $M_{n}\left( X\right) $. Consider also the formula $\psi _{\vec{c}}\left(
x_{1},\ldots ,x_{k-1}\right) $ defined by%
\begin{equation*}
\max \left\{ \sup_{\left( \alpha _{0},\ldots ,\alpha _{k-1}\right) \in X_{%
\vec{c},n}}1\dotminus\left\Vert \sum_{i=0}^{k-1}\alpha _{i}\otimes
x_{i}\right\Vert ,\mu _{\vec{c}}\left( \vec{x}\right) \right\} 
\end{equation*}%
and the formula $\eta \left( x_{1},\ldots ,x_{k-1}\right) $ defined by 
\begin{equation*}
\sup_{\left( \alpha _{0},\ldots ,a_{k-1}\right) \in X_{\vec{a},n}}\max
\left\{ \left\Vert \sum_{i=0}^{k-1}\alpha _{i}\otimes x_{i}\right\Vert %
\dotminus1,1\dotminus\left\Vert \sum_{i=0}^{k-1}\alpha _{1}\otimes
x_{i}\right\Vert \right\} 
\end{equation*}%
where $x_{0}=1$ and $X_{\vec{a},n}$ is the set of tuples $(\alpha
_{0},\ldots ,\alpha _{k-1})$ in $M_{n}$ such that%
\begin{equation*}
\left\Vert \sum_{i=0}^{k-1}\alpha _{i}\otimes a_{i}\right\Vert =1\text{.}
\end{equation*}%
Let $\theta _{\vec{c}}\left( x_{1},\ldots ,x_{k-1}\right) $ be the formula%
\begin{equation*}
\eta \left( x_{1},\ldots ,x_{k-1}\right) +100Lk\psi _{\vec{c}}\left(
x_{1},\ldots ,x_{k-1}\right) \text{.}
\end{equation*}%
We claim that the collection of formulas $\theta _{\vec{c}}$ satisfies the
conclusion of the statement. Since
the space of formulas with parameters from various $M_{n}$'s is separable,
we can then just replace the formulas $\theta _{\vec{c}}$ with a countable
dense set to obtain a sequence of formulas as in the statement.

Suppose thus that $X$ is a nuclear operator system and $\vec{b}$ is a $k$%
-tuple of elements of $X$ with $b_{0}=0$. Let $\phi :E\rightarrow X$ be the
unital linear map such that $\phi \left( a_{i}\right) =b_{i}$ for $i\leq k-1$%
.\ Suppose that $n\geq q$ and $\vec{c}$ is a $k$-tuple in $M_{n}$ with $%
c_{0}=1$ such that $\theta _{\vec{c}}(\vec{b})\leq \frac{1}{4}$. Fix $\delta
\in (0,\frac{1}{4})$ and $\varepsilon \in (0,\frac{1}{400kL})$ such that $%
\eta (\vec{b})<\delta $ and $\psi _{\vec{c}}(\vec{b})<\varepsilon $. It
follows from $\eta (\vec{b})<\delta $ that $\left\Vert \phi \right\Vert
_{n}<1+\delta $. Furthermore, since $\delta <\frac{1}{4}$, we have that $%
\phi $ is invertible and $\left\Vert \phi ^{-1}\right\Vert _{n}=\left\Vert
\phi ^{-1}\right\Vert _{cb}<1+\delta $. Since $\left\Vert a_{i}^{\prime
}\right\Vert \leq L$, a straightforward computation shows that $\left\Vert
b_{i}^{\prime }\right\Vert \leq 2L$ for every $i\leq k-1$, where $\left(
b_{0}^{\prime },\ldots ,b_{k-1}^{\prime }\right) $ is the dual basis of $%
\vec{b}$. It remains to show that $\left\Vert \phi \right\Vert
_{cb}<1+\delta +100kL\varepsilon $. Denote by $\gamma :\mathrm{span}(\vec{b}%
)\rightarrow M_{n}$ the unital linear map such that $\gamma \left(
b_{i}\right) =c_{i}$. Observe that, since $\psi _{\vec{c}}(\vec{b}%
)<\varepsilon $, we have that $\left\Vert \gamma \right\Vert
_{cb}<1+\varepsilon $ and there exists a completely positive map $\rho
_{0}:M_{n}\rightarrow X$ such that 
\begin{equation*}
\left\Vert \rho _{0}\circ \gamma -\iota \right\Vert _{n}<\varepsilon ,
\end{equation*}%
where $\iota $ is the inclusion map of $\mathrm{span}(\vec{b})$ inside $X$.
Define $\rho :M_{n}\rightarrow X$ by 
\begin{equation*}
\rho _{0}\left( z\right) =\rho \left( z\right) +\tau \left( z\right) \left(
1-\rho \left( 1\right) \right) ,
\end{equation*}%
where $\tau $ is the normalized trace of $M_{n}$, and observe that $\rho $
is a unital completely positive map such that $\left\Vert \rho -\rho
_{0}\right\Vert _{cb}\leq \left\Vert 1-\rho \left( 1\right) \right\Vert
<\varepsilon $. Observe now that 
\begin{equation*}
\left\Vert \gamma \circ \phi \right\Vert _{cb}=\left\Vert \gamma \circ \phi
\right\Vert _{n}\leq \left( 1+\delta \right) \left( 1+2\varepsilon \right)
\leq 1+\delta +4\varepsilon 
\end{equation*}%
and hence $\left\Vert \rho \circ \gamma \circ \phi \right\Vert _{cb}\leq
1+\delta +4\varepsilon $. Since $\left\Vert \rho \circ \gamma \circ \phi
-\phi \right\Vert _{n}<\varepsilon $, the small perturbation argument \cite[%
Lemma 2.13.2]{pisier_introduction_2003} shows that $\left\Vert \phi
\right\Vert _{cb}<1+\delta +100kL\varepsilon $.

Suppose now that $\phi :E\rightarrow X$ is an invertible unital linear map
such that $\left\Vert \phi \right\Vert _{cb}<1+\delta $ and $\left\Vert \phi
^{-1}\right\Vert _{cb}<1+\delta $. Fix $\delta ^{\prime },\varepsilon >0$
such that $\delta ^{\prime }+100kL\varepsilon <\delta $ and $\max \left\{
\left\Vert \phi \right\Vert _{cb},\left\Vert \phi ^{-1}\right\Vert
_{cb}\right\} <\delta ^{\prime }$. Set $\vec{b}=\phi (\vec{a})$. Since $X$
is nuclear, there exist $n\in \mathbb{N}$ and ucp linear maps $\gamma
:X\rightarrow M_{n}$ and $\rho :M_{n}\rightarrow X$ such that 
\begin{equation*}
\left\Vert \left( \rho \circ \gamma -id_{X}\right) _{|\phi \left[ E\right]
}\right\Vert _{n}<\varepsilon \text{.}
\end{equation*}%
Therefore $\theta _{\vec{c}}(\vec{b})\leq \delta ^{\prime }+100kL\varepsilon
<\delta $.
\end{proof}
\begin{thm}
Suppose that $X$ is a separable nuclear operator system that is elementarily
equivalent to $\mathbb{GS}$. Then $X\ $is completely order isomorphic to $%
\mathbb{GS}$.
\end{thm}

\begin{proof}
Fix $q\in \mathbb{N}$ and $E\subset M_{q}$ be a subsystem. Suppose that $%
\vec{a}$ is a normalized basis of $M_{q}$ with $a_{0}=1$ such that $%
a_{0},\ldots ,a_{k-1}$ is a basis of $E$. Suppose that $\phi :E\rightarrow X$
is a complete order embedding, $\vec{b}=\left( \phi \left( a_{0}\right)
,\ldots ,\phi \left( a_{k-1}\right) \right) $, and $\varepsilon \in (0,\frac{%
1}{4}]$. By \cite[Proposition 4.2]{lupini_universal_2014} and Lemma \ref%
{Lemma:perturb-nonunital}, in order to prove that $X$ is completely order
isomorphic to $\mathbb{GS}$, it is enough to show that there exists a linear
map $\psi :M_{q}\rightarrow X$ such that%
\begin{equation}
\max \{\left\Vert \psi _{|E}-\phi \right\Vert ,\left\Vert \psi \right\Vert
_{q}-1,\left\Vert \psi ^{-1}\right\Vert _{q}-1\}<10^{3}k\varepsilon ^{\frac{1%
}{2}}\text{.\label{Equation:approx-ext}}
\end{equation}%
Let $\theta _{m}\left( x_{1},\ldots ,x_{k-1}\right) $ for $m\in \mathbb{N}$
be the formulas obtained from $\left( a_{0},\ldots ,a_{k-1}\right) $ as in
Lemma \ref{Lemma:elementary-norm}. Since $X$ is nuclear, there exists $m\in 
\mathbb{N}$ such that $\theta _{m}^{X}(\vec{b})<\varepsilon $. Set $\theta
=\theta _{m}$. Consider the formula $\eta (y_{1},\ldots ,y_{q^{2}-1})$
defined by%
\begin{equation*}
\sup_{(\alpha _{0},\ldots ,a_{q^{2}-1})\in X_{\vec{a},q}}\max \left\{
\left\Vert \sum_{i=0}^{q^{2}-1}\alpha _{i}\otimes y_{i}\right\Vert \dotminus%
1,1\dotminus\left\Vert \sum_{i=0}^{q^{2}-1}\alpha _{1}\otimes
y_{i}\right\Vert \right\} 
\end{equation*}%
where $y_{0}=1$ and $X_{\vec{a},q}$ is the set of $q^{2}$-tuples $(\alpha
_{0},\ldots ,\alpha _{q^{2}-1})$ in $M_{q}$ such that%
\begin{equation*}
\left\Vert \sum_{i=0}^{q^{2}-1}\alpha _{i}\otimes a_{i}\right\Vert =1\text{.}
\end{equation*}%
Consider also the formula $\tau (x_{1},\ldots ,x_{k-1},y_{1},\ldots
,y_{q^{2}-1})$ defined by%
\begin{equation*}
\inf_{y_{1},\ldots ,y_{q^{2}-1}}\max \left\{ \sup_{\left( \alpha _{0},\ldots
,\alpha _{k-1}\right) \in X_{\vec{a},n}}\left\Vert \sum_{i=0}^{k-1}\alpha
_{i}\otimes \left( x_{i}-y_{i}\right) \right\Vert ^{2}\dotminus%
10^{4}k^{2}\theta \left( \vec{x}\right) ,\eta (\vec{y})\right\} 
\end{equation*}%
where $x_{0}=y_{0}=1$ and $X_{\vec{a},n}$ is the set of $k$-tuples $(\alpha
_{0},\ldots ,\alpha _{k-1})$ in $M_{n}$ such that%
\begin{equation*}
\left\Vert \sum_{i=0}^{k-1}\alpha _{i}\otimes a_{i}\right\Vert =1\text{.}
\end{equation*}%
Let now $\sigma $ be the sentence%
\begin{equation*}
\sup_{x_{1},\ldots ,x_{k-1}}\min \left\{ \theta \left( \vec{x}\right) %
\dotminus\frac{1}{4},\inf_{y_{1},\ldots ,y_{q^{2}}-1}\tau (\vec{x},\vec{y}%
)\right\} \text{.}
\end{equation*}%
Observe that $\sigma^{\mathbb{GS}}=0$ by Lemma \ref%
{Lemma:elementary-norm} and the homogeneity property of $\mathbb{GS}$ given
by \cite[Theorem 4.4]{lupini_universal_2014}. Since $X$ is
elementarily equivalent to $\mathbb{GS}$, we have that $\sigma^X=0 $. Therefore, there exists a tuple $\vec{c}$ in $X\ $such that $\tau (\vec{b},%
\vec{c})<\varepsilon $. Let now $\psi :M_{q}\rightarrow X$ be the unital
linear map such that $\psi \left( a_{i}\right) =c_{i}$ for $i<q^{2}$. It
follows from the fact that $\tau (\vec{b},\vec{c})<\varepsilon $ that $\psi $
satisfies Equation \eqref{Equation:approx-ext}.
\end{proof}

In \cite{goldbring_kirchbergs_2014}, it is asked whether $\mathcal{O}_{2}$
is the only exact model of its theory; we also do not know if $\mathbb{GS}$
is the only $1$-exact model of its theory.

\section{Existentially closed C*-algebras}

As mentioned earlier, it was proven by the second-named author in \cite[\S 4.6]{lupini_universal_2014} that $\mathbb{GS}$ is not completely order isomorphic to a C*-algebra.  Here we generalize this result by showing that no unital C*-algebra is existentially closed as an operator system.

\begin{lemma}\label{lemma1}
Suppose that $\phi :X\rightarrow Y$ is a complete order embedding between operator systems.  Further suppose that $X$ is existentially closed  and $u\in X$
is a unitary.  Then $\phi \left( x\right) $ is a unitary.
\end{lemma}

\begin{proof}
Suppose that $n\in \mathbb{N}$ and consider the formula $\varphi \left(
u,x\right) $ defined by%
\begin{equation*}
\min \left\{ \left\Vert 
\begin{bmatrix}
u\otimes I_{n} & x%
\end{bmatrix}%
\right\Vert ^{2},\left\Vert 
\begin{bmatrix}
u\otimes I_{n} \\ 
x%
\end{bmatrix}%
\right\Vert ^{2}\right\} -\left\Vert x\right\Vert ^{2}\text{.}
\end{equation*}%
Observe that 
\begin{equation*}
\left( \inf_{\left\Vert x\right\Vert \leq 1}\varphi \left( u,x\right)
\right) ^{X}=2
\end{equation*}%
by \cite[Theorem 2.4]{blecher_metric_2011}. Therefore%
\begin{equation*}
\left( \inf_{\left\Vert x\right\Vert =1}\varphi \left( \phi \left( u\right)
,x\right) \right) ^{Y}=2\text{,}
\end{equation*}%
whence $\phi \left( u\right) $ %
 is a unitary of $Y$.
\end{proof}

A first draft of this paper contained a proof of the next lemma.  We thank Thomas Sinclair for pointing out to us that this lemma follows immediately from Pisier's Linearization Trick (see, for example, \cite[Theorem 19]{ozawa_about_2013}).

\begin{lemma}\label{lemma2}
Suppose that $\phi:A\to B$ is a ucp map between unital C*-algebras that maps unitaries to unitaries.  Then $\phi$ is a $\ast$-homomorphism.
\end{lemma}


We thank Thomas Sinclair for providing a proof for the following lemma.

\begin{lemma}\label{lemma3}
Suppose that $A$ is a unital C*-algebra and $\dim(A)>1$.  Then there is a unital C*-algebra $B$ and a complete order embedding $\phi:A\to B$ that is \emph{not} a $\ast$-homomorphism.
\end{lemma}

\begin{proof}

We first remark that $A$ has a nonpure state.  Indeed, since the states separate points and every state is a linear combination of pure states, we have that the pure states separate points.  Since $\dim(A)>1$, this implies that there are at least two pure states, whence any proper convex combination of these two pure states is nonpure.

Secondly, we remark that a nonpure state on $A$ is not multiplicative.  Indeed, if $\phi$ is a proper convex combination of the distinct pure states $\phi_1$ and $\phi_2$, then taking a unitary $u$ on which $\phi_1$ and $\phi_2$ differ, we have that $\phi(u)$ has modulus strictly smaller than $1$.

We are now ready to prove the lemma.  Suppose that $A$ is concretely represented as a subalgebra of $B(H)$.  Let $\phi$ be a non-pure state.  Then the map $$x\mapsto (\phi(x)\cdot 1)\oplus x:A\to B(H\oplus H)$$ is a complete order embedding that is not a $\ast$-homomorphism.
\end{proof}

\begin{cor}\label{C*neverecopsystem}
No unital C*-algebra is existentially closed as an operator
system.
\end{cor}

\begin{proof}
This follows immediately from Lemmas \ref{lemma1}, \ref{lemma2}, and \ref{lemma3} (noting that existentially closed operator systems are infinite-dimensional).
\end{proof}

\begin{rmk}
Lemma \ref{lemma1} remains valid in the operator space category as well (with an identical proof).  As a consequence, we see that if $Z$ is an existentially closed operator space, then $Z$ has no unitaries.  Indeed, if $Z$ is concretely represented as a subspace of $\B(H)$, then the map $$x\mapsto x\oplus 0:Z\to \B(H\oplus H)$$ is a complete isometric embedding into a C*-algebra whose image contains no  unitaries, whence, by Lemma 1, $Z$ cannot contain any unitaries. In particular, we see that $\mathbb{NG}$ contains no unitaries, a fact already observed (implicitly) in \cite[Proposition 3.2]{oikhberg_non-commutative_2006}.
\end{rmk}

\begin{rmk}
Corollary \ref{C*neverecopsystem} in particular shows that no unital exact C*-algebra $A$ is existentially closed as an operator system.  We can be a bit more precise about how $A$ fails to be existentially closed as an operator system.  Indeed, since $A$ is exact, by universality, there is a complete order embedding $A\hookrightarrow \mathbb{GS}$.  We claim that this embedding is not existential.  Indeed, since $\mathbb{GS}$ is existentially closed, if the above embedding were existential, then $A$ would be existentially closed as an operator system, contradicting Corollary \ref{C*neverecopsystem}.
\end{rmk}

%

Given the above discussion, the following question seems natural:

\begin{question}\label{C*elem}
Is the class of operator systems unitally completely order isomorphic to a C*-algebra an elementary class?
\end{question}

We now give a condition that would ensure a positive answer to Question \ref{C*elem}.  Suppose that $(X_i \ : \ i\in I)$ is a family of operator systems and $\mathcal{U}$ is an ultrafilter on $I$.  If $u_i\in X_i$ is a unitary for each $i$, then it is clear that $(u_i)^\bullet\in \prod_\mathcal{U} X_i$ is a unitary of $\prod_{\mathcal U} X_i$.

\begin{question}\label{unitarydefinable}
With the preceding notation, if $u$ is a unitary in $\prod_\mathcal{U} X_i$, are there unitaries $u_i\in X_i$ for which $u=(u_i)^\bullet$?
\end{question}

We should note that the analog of Question \ref{unitarydefinable} for C*-algebras has a positive answer (see \cite{farah_model_2015}).

\begin{prop}
If Question \ref{unitarydefinable} has a positive answer, then Question \ref{C*elem} has a positive answer.
\end{prop}

\begin{proof}
Clearly the class of operator systems completely order isomorphic to a C*-algebra is closed under isomorphisms and ultraproducts.  It suffices to check that it is closed under ultraroots.  Towards this end, suppose that $X$ is an operator system for which $X^\mathcal{U}$ is a C*-algebra; we need to show that $X$ is a C*-algebra.  It suffices to show that $X$ is closed under multiplication.  We first show that the product of any two unitaries in $X$ remains in $X$.  Suppose that $u,v\in X$ are unitaries.  By \cite{blecher_metric_2011}, $uv\in X$ if and only if the matrix $ \left[ \begin{matrix}1 & u \\ v & x\end{matrix}\right]$ is $\sqrt{2}$ times a unitary of $M_2(X)$.  However, the aforementioned matrix is $\sqrt{2}$ times a unitary $A$ of $M_2(X^\mathcal{U})$; by assumption, $A=(A_n)^\bullet$, where each $A_n$ is a unitary in $M_2(X)$.  Since unitaries in an operator space form a closed set, we have the desired conclusion.

In order to finish the proof, it suffices to prove that the linear span of the unitaries in $X$ are dense in $X$.  Towards this end, fix $x\in X$ with $\|x\|\leq \frac{1}{2}$.  By \cite[\S II.3.2.16]{blackadar_operator_2006}, there are unitaries $u_1,\ldots,u_5\in X^\mathcal{U}$ for which $x=\frac{1}{5}(u_1+\cdots +u_5)$.  By assumption, we may write each $u_i=(u_{i,n})^\bullet$, where each $u_{i,n}$ is a unitary of $X$.  It follows that some subsequence of $(\frac{1}{5}(u_{1,n}\cdots+u_{5,n}))$ converges to $x$.
\end{proof}

\bibliographystyle{amsplain}
\bibliography{bib-NG}

\end{document}